\documentclass[11pt]{amsart}
\usepackage{amsmath, amssymb, verbatim,color, amscd}

\newtheorem{theorem}{Theorem}[section]
\newtheorem{lemma}[theorem]{Lemma}

\newtheorem{corollary}[theorem]{Corollary}

\renewcommand{\leq}{\leqslant}
\renewcommand{\geq}{\geqslant}

\theoremstyle{definition}

\newtheorem{example}[theorem]{Example}

\theoremstyle{definition}
\newtheorem{remark}[theorem]{Remark}
\numberwithin{equation}{section}

\renewcommand{\leq}{\leqslant}
\renewcommand{\geq}{\geqslant}

\numberwithin{equation}{section} \numberwithin{figure}{section}

\title[Completion of the moduli space for polarized CY manifolds]{Completion of the moduli space for polarized Calabi-Yau manifolds}
 \author[Y. Zhang]{Yuguang Zhang}
 \thanks{The  author is supported in part by  grant NSFC-11271015.}
\address{Mathematical Sciences Center,  Tsinghua University,  Beijing 100084, P.R.China.}
\email{yuguangzhang76@yahoo.com}

\begin{document}
\begin{abstract}
 In this paper, we construct a  completion of the moduli space for polarized Calabi-Yau manifolds by using  Ricci-flat K\"{a}hler-Einstein metrics and  the Gromov-Hausdorff topology, which parameterizes certain Calabi-Yau varieties.  We    then  study the algebro-geometric perperties and the Weil-Petersson geometry of such completion. We show that the completion can be exhausted by sequences of quasi-projective varieties, and  new points added have finite Weil-Petersson distance to the interior.
\end{abstract}
\maketitle
\section{Introduction}
A Calabi-Yau manifold $X$   is a   simply connected complex  projective manifold with trivial canonical bundle $\varpi_{X}\cong \mathcal{O}_{X}$,  and  a polarized Calabi-Yau manifold $(X,L)$ is a Calabi-Yau manifold $X$ with an ample line bundle $L$.
Let $\mathcal{M}^{P}$ be the moduli space of polarized Calabi-Yau manifolds $(X,L)$ of dimension $n$ with a fixed Hilbert polynomial $P=P(\mu)=\chi (X,L^{\mu})$, i.e.  $$ \mathcal{M}^{P}=\big\{(X,L)|P(\mu)=\chi (X,L^{\mu}) \big\}/ \sim,$$ where $(X_{1},L_{1})\sim(X_{2},L_{2})$ if and only if  there is an isomorphism $\psi: X_{1}  \rightarrow X_{2}$ such that $L_{1}=\psi^{*}L_{2}$.  We denote the equivalent class $[X,L]\in  \mathcal{M}^{P}$ represented by $(X,L)$.

 The  compactifications of   moduli spaces were   studied in various cases, for example, the Mumford's compactification of moduli spaces for curves (cf.  \cite{Mum}),  the Satake compactification of moduli spaces for Abelian varieties  (cf. \cite{Sat}), and more  recently the compact moduli spaces  for general type  stable varieties of higher dimension  (cf. \cite{Kol}).   Because of the importance of Calabi-Yau manifolds in mathematics and physics (cf. \cite{Yau2}), it is  also  desirable to have  compactifications of $\mathcal{M}^{P}$.
  The purpose of this paper is to construct a completion  of  $\mathcal{M}^{P}$  in a certain sense, which can be viewed as a  partial compactification.

There are several perspectives  towards   this moduli space  $\mathcal{M}^{P}$. First of all, the  Bogomolov-Tian-Todorov's  unobstructedness  theorem  of Calabi-Yau manifolds implies that $\mathcal{M}^{P}$ is a complex orbifold (cf. \cite{Ti,Tod}).  The variation of   Hodge structures  gives a natural orbifold  K\"{a}hler metric on $\mathcal{M}^{P}$,
called the Weil-Petersson metric,   which  is the curvature of the first Hodge bundle    with a natural  Hermitian metric  (cf. \cite{Ti}). Other natural metrics were also studied in  \cite{Lu} and \cite{Wang2} etc.
From the algebro-geometric point of view,    Viehweg proved   in \cite{Vie}  that
$\mathcal{M}^{P}$ is a  quasi-projective variety, and coarsely represents the moduli functor  $\mathfrak{M}^{P}$ for polarized Calabi-Yau manifolds with   Hilbert polynomial $P$.
The third perspective  is   to understand $\mathcal{M}^{P}$ by  considering  Ricci-flat K\"{a}hler-Einstein metrics.

 For a polarized Calabi-Yau manifold $(X,L)$, Yau's theorem on  the Calabi conjecture, so  called Calabi-Yau theorem, asserts that there exists a unique Ricci-flat K\"{a}hler-Einstein metric $\omega$ with $\omega\in c_{1}(L)$, i.e. the Ricci curvature ${\rm Ric}(\omega)\equiv 0$  (cf. \cite{Ya}).   This  theorem is obtained by showing the existence and the  uniqueness of the solution of  the  Monge-Amp\`ere equation \begin{equation}\label{MA0} (\omega_{0}+\sqrt{-1}\partial\overline{\partial }\varphi)^{n}=(-1)^{\frac{n^{2}}{2}}\Omega\wedge\overline{\Omega}, \  \  \ \sup_{X} \varphi=0,  \end{equation}  for any background K\"{a}hler metric $\omega_{0}\in c_{1}(L)$, and letting  $\omega=\omega_{0}+\sqrt{-1}\partial\overline{\partial }\varphi$,     where   $\Omega$ is a  holomorphic volume form, i.e. a nowhere vanishing section of    $\varpi_{X}$.  We can regard $\mathcal{M}^{P}$ as a parameter space of Ricci-flat Calabi-Yau manifolds.

   In \cite{Grom}, Gromov introduced   the  notion of Gromov-Hausdorff distance between metric spaces, which provides a frame to study families  of Riemannian manifolds. For any two compact metric spaces $A$ and $B$, the Gromov-Hausdorff distance of $A$ and $B$ is  $$d_{GH}(A,B)=\inf\{d_{H}^{Z}(A,B)| \exists \  {\rm isometric \ embeddings} \ A, B\hookrightarrow Z\}, $$ where $Z$ is a metric space,  $d_{H}^{Z}(A,B)$ is the standard  Hausdorff distance between $A$ and $B$ regarded as subsets by the isometric embeddings, and the infimum is taken for all possible $Z$ and isometric  embeddings. We denote $\mathcal{M}et$   the space of the isometric equivalence   classes of  all compact metric spaces equipped with the  topology, called the Gromov-Hausdorff topology,  induced by the Gromov-Hausdorff distance $d_{GH}$. Then $\mathcal{M}et$ is a complete metric space (cf. \cite{Grom,Rong}).   The Gromov-Hausdorff topology was used to study moduli spaces for Einstein metrics by various authors (See  for instance \cite{An1,CT,CT2} etc.).

    The Calabi-Yau theorem gives  a continuous  map \begin{equation}\label{map0} \mathcal{CY}: \mathcal{M}^{P}  \rightarrow \mathcal{M}et,  \  \  {\rm by}  \  \  \left[X,L\right]  \mapsto  (X, \omega),\end{equation}
 where $\omega$ is the unique Ricci-flat K\"ahler-Einstein metric representing $c_{1}(L)$. However, $ \mathcal{CY}$ is not injective in general  since $\mathcal{M}^{P}$ contains the information of complex structures.

 For constructing compactifications of $\mathcal{M}^{P}$, Yau suggested   that one uses the Weil-Petersson metric to obtain a metric  completion of   $\mathcal{M}^{P}$   first, and then tries  to compactify this completion (cf. \cite{Lu}). In \cite{Wang2}, an alternative approach   is proposed by using the Gromov-Hausdorff distance,   instead of the Weil-Petersson metric,   to construct a completion.
  Let  $\overline{\mathcal{CY}( \mathcal{M}^{P} )}$ be the  closure of the image  $\mathcal{CY}( \mathcal{M}^{P} ) $ in $\mathcal{M}et$.  There is a natural metric space structure on $\overline{\mathcal{CY}( \mathcal{M}^{P} )}$ by restricting the Gromov-Hausdorff distance.
    A question  is  to understand $\overline{\mathcal{CY}( \mathcal{M}^{P} )}$ from  the     algebraic geometry  and   the  Weil-Petersson geometry  viewpoints.

     A normal projective variety $X$ is called  $1$-Gorenstein if  the dualizing   sheaf $\varpi_{X}$ is an invertible  sheaf, i.e. a line bundle,  and is called Gorenstein if furthermore $X$ is Cohen-Macaulay. A variety $X$ has only canonical singularities if $X$ is $1$-Gorenstein, and for any resolution $\bar{\pi}: \bar{X}\rightarrow X$, $\bar{\pi}_{*}\varpi_{\bar{X}}=\varpi_{X}$, which is equivalent to that the canonical divisor $\mathcal{K}_{X}$ is Cartier, and $$\mathcal{K}_{\bar{X}} = \bar{\pi}^{*}\mathcal{K}_{X}+\sum\limits_{E} a_{E}E, \ \ {\rm and} \  \  a_{E} \geq 0, $$ where $E$ are  exceptional prime divisors.  If $X$ has only canonical singularities, then the singularities are rational, and $X$ is Cohen-Macaulay (cf. (C)   of \cite[Section 3]{Re}), which implies that $X$ is Gorenstein.  A Calabi-Yau variety $X$ is a normal projective Gorenstein variety with trivial dualizing   sheaf $\varpi_{X}\cong \mathcal{O}_{X}$, and having at most  canonical singularities. A  polarized Calabi-Yau variety $(X,L)$ is a Calabi-Yau variety $X$  with an ample line bundle $L$.

  If $(Y, d_{Y})\in \overline{\mathcal{CY}( \mathcal{M}^{P} )}$, then there is a sequence  $\{[X_{k}, L_{k}]\}\subset \mathcal{M}^{P}$ such that  $\mathcal{CY}([X_{k}, L_{k}])=(X_{k}, \omega_{k})$ converge to $(Y, d_{Y})$ in the Gromov-Hausdorff sense.  Note that the diameters and the volumes satisfy that   $${\rm diam}_{\omega_{k}}(X_{k})\rightarrow {\rm diam}_{d_{Y}}(Y), \  \  \ {\rm Vol}_{\omega_{k}}(X_{k})=\frac{1}{n!}c_{1}(L_{k})^{n}\equiv {\rm cont.}$$ independent of $k$, which imply that    $Y$ is  a  non-collapsed limit.  In  \cite{DS}, Donaldson and Sun studied  the algebro-geometric structure of $Y$, and proved      that $Y$ is homeomorphic to a Calabi-Yau variety $X_{0}$ of dimension $n$. Hence loosely speaking,  $ \overline{\mathcal{CY}( \mathcal{M}^{P} )}$ can be regarded  as a parameter space of certain Calabi-Yau varieties.

 A degeneration of polarized Calabi-Yau manifolds $(\pi_{\Delta}: \mathcal{X}\rightarrow \Delta, \mathcal{L})$ is a flat morphism from a  variety $\mathcal{X}$ of dimension $n+1$  to a disc $ \Delta\subset \mathbb{C}$ such that for any $t\in \Delta^{*}=\Delta \backslash\{0\}$,  $X_{t}=\pi_{\Delta}^{-1}(t)$ is a Calabi-Yau manifold, the central fiber   $X_{0}=\pi_{\Delta}^{-1}(0)$ is singular,  and $\mathcal{L}$  is a relative  ample line bundle  on $\mathcal{X}$.  If  we further  assume that  $X_{0}$ is  a Calabi-Yau  variety, then   the total space $\mathcal{X}$ is normal as  any fiber $X_{t}$ is reduced and normal. Thus
      the relative dualizing sheaf $\varpi_{\mathcal{X}/ \Delta}$  is defined, i.e. $\varpi_{\mathcal{X}/ \Delta}\cong \varpi_{\mathcal{X}}\otimes \pi_{\Delta}^{*}\varpi_{\Delta}^{-1}$,  and is  trivial, i.e. $\varpi_{\mathcal{X}/ \Delta}\cong \mathcal{O}_{\mathcal{X}}$,   since every fiber is normal,  Cohen-Macaulay and Gorenstein.

Let  $(\pi_{\Delta}: \mathcal{X}\rightarrow \Delta, \mathcal{L})$ be a degeneration of polarized Calabi-Yau manifolds with a Calabi-Yau variety $X_{0}$ as the central fiber, and  $\omega_{t}$ be the unique Ricci-flat K\"{a}hler-Einstein metric  on $X_{t}$ representing $c_{1}(\mathcal{L}|_{X_{t}})$, $t\in \Delta^{*}$.  The asymptotic behaviour of $\omega_{t}$ when $t\rightarrow 0$ is studied in \cite{RZ,RZG}, and it is shown that $(X_{t}, \omega_{t})$ converges to a compact metric space of the same dimension   in the Gromov-Hausdorff sense.  This result, together with Donaldson-Sun's theorem,  shows the equivalence between the  algebro-geometric degenerating  Calabi-Yau manifolds to a Calabi-Yau variety and the non-collapsing Gromov-Hausdorff convergence of Ricci-flat K\"{a}hler-Einstein metrics.

 The first goal of the present paper is to investigate  the algebro-geometric structure of   $\overline{\mathcal{CY}( \mathcal{M}^{P} )}$.

  \begin{theorem}\label{main} There is a Hausdorff topological space $\overline{\mathcal{M}}^{P}$, and    a surjection  $$\overline{\mathcal{CY}}:\overline{\mathcal{M}}^{P} \rightarrow   \overline{\mathcal{CY}( \mathcal{M}^{P} )}$$    satisfying  the follows.    \begin{itemize}
  \item[i)]   $\mathcal{M}^{P}$ is  an open dense subset of  $ \overline{\mathcal{M}}^{P}$,   and $\overline{\mathcal{CY}}|_{\mathcal{M}^{P}}=\mathcal{CY}$.
   \item[ii)]  For any $p\in  \overline{\mathcal{M}}^{P}$,   $\overline{\mathcal{CY}}(p)$ is homeomorphic to a Calabi-Yau variety.
     \item[iii)]  There is an exhaustion  $$ \mathcal{M}^{P}\subset  \mathcal{M}_{m(1)} \subset  \mathcal{M}_{m(2)}  \subset  \cdots \subset \mathcal{M}_{m(l)}   \subset \cdots \subset \overline{\mathcal{M}}^{P}=\bigcup_{l\in\mathbb{N}}  \mathcal{M}_{m(l)},$$ where $m(l)\in\mathbb{N}$ for any $l\in\mathbb{N}$,  such that   $\mathcal{M}_{m(l)} $ is a  quasi-projective variety, and there is an ample line bundle $\lambda_{m(l)}$ on   $\mathcal{M}_{m(l)} $.
           \item[iv)]  Let  $(\pi_{\Delta}: \mathcal{X}\rightarrow \Delta, \mathcal{L})$ be a degeneration of polarized  Calabi-Yau manifolds  with   a Calabi-Yau variety $X_{0}$ as the central fiber.  Assume that for any $t\in\Delta^{*}$, there is an ample line bundle $L_{t}$ on $X_{t}$ such that $L_{t}^{k}\cong \mathcal{L}|_{X_{t}}$ for a  $k\in\mathbb{N}$, and $[X_{t},L_{t}]\in \mathcal{M}^{P}$.
      Then there is    a unique morphism $\rho: \Delta \rightarrow \mathcal{M}_{m(l)} $,  for  $l\gg 1$,  such that  $\overline{\mathcal{CY}}(\rho(t))$ is homeomorphic to  $ X_{t}$ for any $t\in\Delta$, and $$ \overline{\mathcal{CY}}(\rho(t)) \rightarrow \overline{\mathcal{CY}}(\rho(0)),$$ when $t\rightarrow 0$, in the Gromov-Hausdorff sense.
      Furthermore,  $\rho^{*}\lambda_{m(l)}=\pi_{\Delta,*} \varpi_{\mathcal{X}/\Delta}^{\nu(l)}$ for a  $\nu(l)\in\mathbb{N}$.

   \end{itemize}
\end{theorem}

 \begin{remark}\label{re}
 In  general, we do not expect  $ \mathcal{M}_{m(l)}  = \overline{\mathcal{M}}^{P}$ for some $m(l)$ because of  the  lack of the  boundedness condition  for singular Calabi-Yau varieties (cf. Section 3 in  \cite{Gro}).
    \end{remark}

When $n=2$, a Calabi-Yau variety is a K3 orbifold, and a degeneration of K3 surfaces to a K3 orbifold is called a degeneration of type I.  It is well-known that one can fill the holes in the moduli space of K\"{a}hler  polarized K3 surfaces by some K\"{a}hler  K3 orbifolds, and  obtain  a complete moduli space
     (cf. \cite{Kob,KT}).
     The relationship between such moduli space and the degeneration of Ricci-flat K\"{a}hler-Einstein metrics  is also established    in \cite{Kob,KT}.  Theorem \ref{main} is  a generalization of \cite{Kob,KT} to higher dimensional polarized  Calabi-Yau manifolds.

In \cite{Ma},   K\"ahler-Einstein metrics are used to construct compactifications of moduli spaces  for K\"ahler-Einstein orbifolds, and it is proved that such compactification  coincides with the standard Mumford's  compactification in the case of  curves.   The moduli space  of Fano manifolds admitting K\"ahler-Einstein metrics is   constructed in a recent preprint  \cite{Od}, which generalizes the earlier work \cite{Ti2} for del Pezzo surfaces.   The Gromov-Hausdorff compactification of such  moduli space for del Pezzo surfaces of each degree  is  studied
         in  \cite{OSS} (see also \cite{MM} for the quartic case),  and it is proven to agree with  certain algebro-geometric compactification.   We can regard Theorem \ref{main} as an analog result  of  \cite{OSS} for the Calabi-Yau case.

     Now we study the Weil-Petersson geometry of $\overline{\mathcal{M}}^{P}$.   Note that for any flat family $(\pi_{\Delta}: \mathcal{X} \rightarrow \Delta, \mathcal{L})$ of polarized Calabi-Yau manifolds with  Hilbert polynomial $P$,  there is a unique morphism $f: \Delta \rightarrow \mathcal{M}^{P}$, since $\mathcal{M}^{P}$ coarsely represents the  moduli functor $\mathfrak{M}^{P}$.   The Weil-Petersson metric $\omega_{WP}$ is an  orbifold K\"ahler metric   on $\mathcal{M}^{P}$ (cf. \cite{Ti})  characterized by   $$ f^{*}\omega_{WP}=-\frac{\sqrt{-1}}{2\pi}\partial\overline{\partial}\log \int_{X_{t}}(-1)^{\frac{n^{2}}{2}} \Omega_{t} \wedge \overline{\Omega}_{t},$$
     where  $\Omega_{t} $ is a relative holomorphic volume form, i.e. a nowhere vanishing  section of $ \varpi_{\mathcal{X}/\Delta}$.
       The metric  $ \omega_{WP}$ is the curvature of the first Hodge bundle  with a natural  Hermitian metric.

       In \cite{CGH}, Candelas,  Green and H\"{u}bsch found some nodal degenerations of Calabi-Yau 3-folds with finite Weil-Petersson distance.  In general,   \cite{Wang1}  shows that  if $(\pi_{\Delta}: \mathcal{X}\rightarrow \Delta, \mathcal{L})$ is    a degeneration of polarized  Calabi-Yau manifolds, and if the central fiber   $X_{0}$ is a Calabi-Yau variety, then the Weil-Petersson distance between $\{0\}$ and the interior $\Delta^{*}$ is finite, i.e.  $\omega_{WP}$ is not complete on $\Delta^{*}$.   Conversely,   if we assume that the Weil-Petersson distance  of $\{0\}$  is finite,  then $\pi_{\Delta}: \mathcal{X}\rightarrow \Delta$ is birational to a degeneration  $\pi_{\Delta}': \mathcal{X}'\rightarrow \Delta$ such that $\mathcal{X}\backslash X_{0}\cong \mathcal{X}'\backslash X_{0}'$, and $X_{0}'$ is a Calabi-Yau variety by recent papers  \cite{To} and \cite{Ta}.  As a consequence,  the  algebro-geometric degenerating  Calabi-Yau manifolds to a Calabi-Yau variety is equivalent  to the finiteness of the Weil-Petersson distance.

         Our next result shows that the  points  in $\overline{\mathcal{M}}^{P}\backslash   \mathcal{M}^{P}$ have finite Weil-Petersson distance.

       \begin{theorem}\label{main2} Let   $\overline{\mathcal{M}}^{P}$ and   $\overline{\mathcal{CY}}$  be the same as in Theorem \ref{main}.
          \begin{itemize}
  \item[i)]   For any point  $x\in  \overline{\mathcal{M}}^{P}\backslash   \mathcal{M}^{P}$,  there is a curve $\gamma$ such that $\gamma(0)=x$,  $\gamma( (0,1] )\subset  \mathcal{M}^{P}$ and the length of $\gamma$  under the Weil-Petersson metric  $\omega_{WP}$ is finite, i.e.   $${\rm length}_{\omega_{WP}}(\gamma) < \infty .$$
    \item[ii)]  Let  $(\pi_{\Delta}: \mathcal{X}\rightarrow \Delta, \mathcal{L})$ be a degeneration of polarized  Calabi-Yau manifolds   such that  for any $t\in\Delta^{*}$, $L_{t}^{k}\cong \mathcal{L}|_{X_{t}}$ for a  $k\in\mathbb{N}$, and $[X_{t},L_{t}]\in \mathcal{M}^{P}$, where $L_{t}$ is an ample line bundle.   If the Weil-Petersson distance between  $0\in \Delta$ and the interior  $\Delta^{*}$ is finite,
      then there is    a unique morphism $\varrho: \Delta \rightarrow \mathcal{M}_{m(l)} $,  for  $l\gg 1$,  such that $\mathcal{CY}(\varrho(t))$ is homeomorphic to  $X_{t}$,  $t\in\Delta^{*}$.
   \end{itemize}
\end{theorem}

This paper is organized as follows.   Section 2 studies Ricci-flat K\"{a}hler-Einstein metrics. In Section 2.1, we recall the generalized Calabi-Yau theorem
  in  \cite{EGZ}, and then in Section 2.2, we use \cite{DS} to  improve the earlier work in \cite{RZG,RZ}, i.e. we show  that along a degeneration of polarized Calabi-Yau manifolds with a Calabi-Yau variety as the central fiber, the Gromov-Hausdorff limit of Ricci-flat K\"{a}hler-Einstein metrics on general fibers is homeomorphic to  the central fiber.  All of properties about the  Gromov-Hausdorff topology in Theorem \ref{main}   are from this section. The technique developed in this section can also  be used to study   the unique filling-in problem  for degenerations of Calabi-Yau manifolds, i.e. Corollary \ref{coro}, which has  independent interests.
     In Section 3, we study the algebraic geometry of the moduli space. Firstly, we recall the  Viehweg's  construction of quasi-projective moduli space for polarized  Calabi-Yau manifolds  (cf. \cite{Vie}) in Section 3.1. Secondly, in Section 3.2, we construct an enlarged moduli space  of $ \mathcal{M}^{P}$ by using the construction of moduli spaces for varieties with at worst canonical singularities  (cf.  Section 8 of  \cite{Vie}).  More precisely, for any $m>0$,  we construct a moduli functor $\mathfrak{M}_{m}$ for polarized   Calabi-Yau varieties  that can be embedded in $\mathbb{CP}^{N}$, $N=N(m)$.  Then we use the results in  Section 8 of  \cite{Vie} to prove that  $\mathfrak{M}_{m}$ can be coarsely represented by a quasi-projective variety. Any $ \mathcal{M}_{m(l)}$ in Theorem \ref{main}  comes from this construction.
  We prove Theorem \ref{main} and Theorem \ref{main2} in Section 4, and finally, we  give a remark for compactifications in Section 5.

In this paper, the notion scheme stands for   separated schemes of finite type  over $\mathbb{C}$, and the notion  variety stands for either  a reduced irreducible scheme  or the set of  its closed points with the natural analytic topology depending on the context. A point in a scheme means  a closed point. For a flat family of schemes $\pi_{T}: \mathcal{X}\rightarrow T$   over $T$, we denote  $X_{t}=\pi_{T}^{-1}(t)$ the fiber $\mathcal{X}\times_{T} \{t\} $ over a point  $t\in T$.   Since a Calabi-Yau manifold  $X$ is defined to be simply connected,  the natural map $\mathcal{M}^{P} \rightarrow \mathcal{M}^{P_{\mu}}$ by $(X,L) \mapsto (X, L^{\mu})$ for any $\mu\in\mathbb{N}$ is injective, where $P_{\mu}(k)=P(\mu k)$,  and thus is an isomorphism. Thus,   we identify  $\mathcal{M}^{P}$ and $\mathcal{M}^{P_{\mu}}$ in this paper.\\

\noindent {\bf Acknowledgements:} The author is   grateful to Prof.  Mark Gross and Prof.   Chenyang Xu   for very useful discussions, especially C.  Xu for  pointing  out the proof of Lemma \ref{l3.1} to him. The author would also like to thank Prof. Valentino Tosatti for sending him the preprints \cite{Bo,Ta} and some comments.
 Part of this work was carried out while the  author was visiting the  Department  of Mathematics,  University of Cambridge, which he  thanks  for the hospitality. Finally, the author  thanks   referees  for some comments.    \\

\section{Ricci-flat K\"{a}hler-Einstein metrics}
In this section, we study  the Gromov-Hausdorff convergence of Ricci-flat K\"{a}hler-Einstein metrics along degenerations of polarized Calabi-Yau manifolds.
\subsection{Singular K\"{a}hler-Einstein metric}
 There is a notion of K\"{a}hler metric for normal varieties (cf.  \cite[Section 5.2]{EGZ}). A smooth  K\"{a}hler metric $\omega$ on a normal variety $X$ is a   usual K\"{a}hler metric on the regular locus $X_{reg}$ such that for any singular point $p\in X$, there is a neighborhood $U_{p}$ with an embedding $ U_{p}\hookrightarrow \mathbb{C}^{N_{p}}$, and a smooth strongly pluri-subharmonic function $\upsilon_{p}$ on  $\mathbb{C}^{N_{p}}$ satisfying $\omega|_{U_{p}\bigcap X_{reg}}=\sqrt{-1}\partial\overline{\partial}\upsilon_{p}|_{U_{p}\bigcap X_{reg}}$.   If these functions  $\upsilon_{p}$ are not smooth, we call $\omega$ a singular K\"{a}hler metric.  A   K\"{a}hler metric $\omega$, possibly    singular,   defines a class $[\omega]$ in $H^{1}(X, \mathcal{PH}_{X})$, where $\mathcal{PH}_{X}$ denotes the sheaf of pluri-harmonic functions on $X$.

  If $L$ is an ample line bundle on $X$,  there is an  $m>0$ such that $L^{m}$ is very ample, and $H^{i}(X, L^{\mu})=\{0\}$ for any $i>0$ and $\mu \geq m$.
  A basis $\Sigma=\{s_{0}, \cdots, s_{N}\}$ of $H^{0}(X, L^{m})$ gives an embedding $\Phi_{\Sigma}:X \hookrightarrow \mathbb{CP}^{N}$ by $x \mapsto [s_{0}(x), \cdots, s_{N}(x)]$, which satisfies  $L^{m}=\Phi_{\Sigma}^{*}\mathcal{O}_{\mathbb{CP}^{N}}(1)$,  where $N=\dim_{\mathbb{C}}H^{0}(X, L^{m})-1$.
 The pullback  $\omega_{\Sigma}=\Phi_{\Sigma}^{*}\omega_{FS}$ of the Fubini-Study metric is a smooth K\"{a}hler metric in the above sense such that $[\omega_{\Sigma}]=m c_{1}(L)\in NS_{\mathbb{R}}(X)$. The Hermitian metric $h_{FS}$ of $\mathcal{O}_{\mathbb{CP}^{N}}(1)$,  whose curvature is the Fubini-Study metric,  restricts to an  Hermitian metric $h_{\Sigma}=\Phi_{\Sigma}^{*}h_{FS}$ on $L^{m}$, which satisfies that $\omega_{\Sigma}=-\frac{\sqrt{-1}}{2}\partial\overline{\partial}\log |\vartheta|_{h_{\Sigma}}^{2}$ on $X_{reg}$ for any nonvanishing local section $\vartheta$ of $L^{m}$.   We   regard $\Phi_{\Sigma}(X)$ as a   point in    $\mathcal{H}il_{N}^{P}$, denoted still by  $\Phi_{\Sigma}(X)$, where   $\mathcal{H}il_{N}^{P}$ is the Hilbert scheme  parametrizing  subschemes of $\mathbb{CP}^{N}$ with the  Hilbert polynomial $P=P(k)=\chi (X,L^{mk})$.

    If  $\Sigma'=\{s_{0}', \cdots, s_{N}'\}$ is  another basis of $H^{0}(X, L^{m})$, we have a matrices   $u=(u_{ij})\in SL(N+1)$ such that $[s_{0}', \cdots, s_{N}']=[\sum\limits_{i=0}^{N}s_{i}u_{i0}, \cdots, \sum\limits_{i=0}^{N}s_{i}u_{iN}],$ denoted by $[\Sigma']=[\Sigma]\cdot u$,  and thus, $\Phi_{\Sigma'}(x)=\sigma(u,\Phi_{\Sigma}(x))$ for any $x\in X$, where $\sigma:  SL(N+1)\times \mathbb{CP}^{N} \rightarrow \mathbb{CP}^{N}$ is the natural $SL(N+1)$-action on $\mathbb{CP}^{N}$. Note that  $\sigma$ induces an   $SL(N+1)$-action on the Hilbert scheme $\mathcal{H}il_{N}^{P}$,  denoted still by $\sigma:  SL(N+1)\times \mathcal{H}il_{N}^{P} \rightarrow \mathcal{H}il_{N}^{P}$.  We have $\Phi_{\Sigma'}(X)=\sigma(u, \Phi_{\Sigma}(X))$, and we   denote   the orbit \begin{equation}\label{eq2.1}O(X,L^{m})=\{\sigma(u, \Phi_{\Sigma}(X))| u\in SL(N+1) \}\subset \mathcal{H}il_{N}^{P}. \end{equation}

 In  \cite{EGZ}, a   generalized Calabi-Yau theorem is obtained for polarized  Calabi-Yau varieties, i.e. the existence and the uniqueness of  singular Ricci-flat K\"{a}hler-Einstein metrics  with bounded potentials.   More precisely, for a polarized Calabi-Yau variety $(X,L)$,   Theorem 7.5 of \cite{EGZ} says that there is a unique bounded function $\varphi$ satisfying the following Monge-Amp\`ere equation \begin{equation}\label{MA} (\omega_{\Sigma}+\sqrt{-1}\partial\overline{\partial }\varphi)^{n}=(-1)^{\frac{n^{2}}{2}}\Omega\wedge\overline{\Omega}, \  \  \ \sup_{X} \varphi=0,  \  {\rm and} \  \varphi\geq - C, \end{equation}   where $\Omega$ is a  holomorphic volume form, i.e. a nowhere vanishing section of the dualizing sheaf   $\varpi_{X}$.
  The restriction of the singular K\"{a}hler metric  $\omega=\omega_{\Sigma}+\sqrt{-1}\partial\overline{\partial }\varphi$ on the regular locus $X_{reg}$  is a smooth  Ricci-flat K\"{a}hler-Einstein metric, and  $\omega\in [\omega_{\Sigma}]=m c_{1}(L)$. Furthermore, $\omega$ is  unique in $m c_{1}(L)$, and particularly is  independent of the choice of $\Phi_{\Sigma}$.    By the boundedness of $\varphi$, we have that $h=\exp(-\varphi) h_{\Sigma}$ is an  Hermitian metric on $L^{m}$ whose curvature is $\omega$.

 We define an  $L^{2}$-norm $\| \cdot \|_{L^{2}(h)}$ on $H^{0}(X,L^{m})$ by
  \begin{equation}\label{normal}\| s\|_{L^{2}(h)}^{2}=\int_{X}|s|^{2}_{h} \omega^{n}=\int_{X}e^{-\varphi}|s|^{2}_{h_{\Sigma}} \omega^{n}.  \end{equation} If $h'$ is another Hermitian metric with the same curvature $\omega$, then $\partial\overline{\partial}\log \frac{h}{h'}\equiv 0$, i.e.   $\log \frac{h}{h'}$ is a pluriharmonic function on a closed normal variety $X$,  and thus $h=e^{\varsigma}h'$ for a constant $\varsigma$.  If $\Sigma_{h}=\{s_{0}, \cdots, s_{N}\}$ is an  orthonormal basis    of $H^{0}(X,L^{m})$ with respect to $\| \cdot \|_{L^{2}(h)}$, then $\Sigma_{h'}=\{e^{-\frac{\varsigma}{2}}s_{0}, \cdots, e^{-\frac{\varsigma}{2}}s_{N}\}$ is  orthonormal  with respect to $\| \cdot \|_{L^{2}(h')}$, and furthermore $\Sigma_{h}$ and $\Sigma_{h'}$ induce the same embedding $ \Phi_{\Sigma_{h}}= \Phi_{\Sigma_{h'}}$.

 If $\Sigma_{h}$ and $\Sigma_{h}'$ are
     two orthonormal bases   of $H^{0}(X,L)$ with respect to $\| \cdot \|_{L^{2}(h)}$, there is an  $u\in SU(N+1)\subset SL(N+1)$ such that  $[\Sigma_{h}]=[\Sigma_{h}']\cdot u$, $\Phi_{\Sigma_{h}'}(x)=\sigma(u,\Phi_{\Sigma_{h}}(x))$ for any $x\in X$, and thus $\Phi_{\Sigma_{h}'}(X)=\sigma(u,\Phi_{\Sigma_{h}}(X))$ in
     $\mathcal{H}il_{N}^{P}$.   The action $\sigma$ and $h$ induce an
       $SU(N+1)$-orbit \begin{equation}\label{eq2.2} RO(X,L^{m})=\{\sigma(u,\Phi_{\Sigma_{h}}(X))| u\in SU(N+1) \}\subset  O(X,L^{m}).  \end{equation} Note that  $ RO(X,L^{m})$  is compact,  and depends only on the the singular K\"{a}hler metric  $\omega$,  but not on the choice of $h$, though   the norm $\| \cdot \|_{L^{2}(h)}$ does.

\subsection{Gromov-Hausdorff convergence of Ricci-flat K\"{a}hler-Einstein metrics}
Let $(\pi_{\Delta}: \mathcal{X}\rightarrow \Delta, \mathcal{L})$ be a
 degeneration of polarized Calabi-Yau manifolds with a Calabi-Yau  variety  $X_{0}$ as the central fiber.
    By taking a certain power of $\mathcal{L}$, we assume that
     $\mathcal{L}$ is relative  very ample, and $R^{i}\pi_{\Delta,*}(\mathcal{L})=\{0\}$ for $i>0$. There is  a morphism  $\tilde{\Phi}:\mathcal{X} \hookrightarrow \mathbb{CP}^{N}\times \Delta \rightarrow \mathbb{CP}^{N}$ by composing an embedding and the projection   such that $\mathcal{L}\cong \tilde{\Phi}^{*} \mathcal{O}_{\mathbb{CP}^{N}}(1)$. In this section, we always assume that any  $\tilde{\Phi}(X_{t})$ does not belong to a proper linear  subspace of $\mathbb{CP}^{N}$ by shrinking $\Delta$ if necessary.   We denote $\omega_{o,t}=\tilde{\Phi}^{*} \omega_{FS}|_{X_{t}}$, and denote  $\omega_{t}$ the unique Ricci-flat K\"{a}hler-Einstein metric in $[\omega_{o,t}]$ for any $t\in \Delta$. Note that   $\omega_{t}=\omega_{o,t}+\sqrt{-1}\partial\overline{\partial}\varphi_{t}$ for a unique bounded  potential function $\varphi_{t}$ with $\sup\limits_{X_{t}}\varphi_{t}=0$, which  satisfies the Monge-Amp\`ere equation (\ref{MA0}) and (\ref{MA}) respectively.

 The limiting  behaviour of $\omega_{t}$, when $t \rightarrow 0$,  is studied intensively in \cite{RuZ},  \cite{RZG} and \cite{RZ}.
 Theorem 1.4 in \cite{RZG} asserts that the diameter has  a uniform  upper bound  $D>0$, i.e.    \begin{equation}\label{diam}{\rm diam}_{\omega_{t}}(X_{t})\leq D, \end{equation} for any $t\in \Delta^{*}$.  Furthermore,  for any smooth family of embeddings $F_{t}:X_{0,reg} \rightarrow X_{t}$ with $F_{0}={\rm Id}$, we have  \begin{equation}\label{metric convergence} F^{*}_{t}\omega_{t} \rightarrow \omega_{0},  \  \  \   \varphi_{t}\circ F_{t}  \rightarrow \varphi_{0}, \  \  \ {\rm and} \  \ \varphi_{t} >-C\end{equation} for a constant $C>0$,  when $t \rightarrow 0$ in the $C^{\infty}_{loc}$-sense, where  $\varphi_{0}$ is the solution of (\ref{MA}), and $\omega_{0}$ is the unique singular Ricci-flat K\"{a}hler-Einstein metric in $c_{1}(\mathcal{L}|_{X_{0}}) $.    In   \cite{RZ}, it is proved that, when $t\rightarrow 0$,  $(X_{t}, \omega_{t})$ converges to a compact metric space $X_{\infty}$ in the Gromov-Hausdorff topology, and $X_{\infty}$  is the metric completion of $(X_{0,reg}, \omega_{0})$.  Actually, $X_{\infty}$ is a Calabi-Yau variety by the following theorem due to Donaldson and Sun  (cf.  \cite{DS,Bo2}).

  \begin{theorem}[Theorem 1.2 of \cite{DS}]\label{D-S}    Let $(X_{k}, L_{k})$ be a sequence  of polarized Calabi-Yau manifolds of dimension $n$ with the same Hilbert polynomial $P$, and $\omega_{k}\in c_{1}(L_{k})$ be the unique Ricci-flat K\"ahler-Einstein metric. We assume that  $${\rm Vol}_{\omega_{k}}(X_{k})=\frac{1}{n!}c_{1}(L_{k})^{n}\equiv\upsilon, \  \   {\rm diam}_{\omega_{k}}(X_{k}) \leqslant D$$ for  constants $D>0$ and $\upsilon>0$, and furthermore,  $(X_{k}, \omega_{k})$ converges to a compact metric space $X_{\infty}$ in the Gromov-Hausdorff sense.  Then we have the follows.
    \begin{itemize}
  \item[i)] $X_{\infty}$ is  homeomorphic to a Calabi-Yau variety, denoted still by $X_{\infty}$.
     \item[ii)]There are constants $m>0$ and $\bar{N}>0$ satisfying the following.  For any $k$,  there is an orthonormal basis $\Sigma_{k}$ of $H^{0}(X_{k}, L_{k}^{m})$ with respect to the $L^{2}$-norm induced by $\omega_{k}$, which induces an embedding $\Phi_{\Sigma_{k}}: X_{k} \hookrightarrow \mathbb{CP}^{\bar{N}}$ with $L_{k}^{m}= \Phi_{\Sigma_{k}}^{*} \mathcal{O}_{\mathbb{CP}^{\bar{N}}}(1)$. And $\Phi_{\Sigma_{k}}( X_{k})$ converges to $X_{\infty}$ in the Hilbert scheme $\mathcal{H}il_{\bar{N}}^{P_{m}}$.
          \item[iii)]  The metric space structure on $X_{\infty}$ is induced by the unique  singular Ricci-flat K\"{a}hler-Einstein metric $\omega \in \frac{1}{m} c_{1}( \mathcal{O}_{\mathbb{CP}^{\bar{N}}}(1)|_{X_{\infty}})$.
       \end{itemize}
   \end{theorem}

By   Proposition 4.15 of  \cite{DS},   $X_{\infty}$ is  a projective normal variety with only log-terminal singularities.  Note that the holomorphic volume forms $\Omega_{k}$ are parallel with respect to $\omega_{k}$, and converge  to a holomorphic volume form $\Omega_{\infty}$ on the regular locus $X_{\infty,reg}$ along the Gromov-Hausdorff convergence by normalizing $\Omega_{k}$ if  necessary. Thus the dualizing sheaf  $\varpi_{X_{\infty}}$ is trivial, i.e.  $\varpi_{X_{\infty}}\cong \mathcal{O}_{X_{\infty}}$, and $X_{\infty}$ is 1--Gorenstein.  Furthermore, the canonical divisor $\mathcal{K}_{X_{\infty}}$ is Cartier and trivial, which implies that $X_{\infty}$ has at worst  canonical singularities.  Then $X_{\infty}$ has only rational singularities,  $X_{\infty}$ is Cohen-Macaulay and is  Gorenstein.  Consequently,  $X_{\infty}$ is  a Calabi-Yau variety.

A natural question is what is the relationship between these two Calabi-Yau varieties  $X_{0}$ and $X_{\infty}$ in our setting.

   \begin{lemma}\label{pro0.1}  Let $(\pi_{\Delta}: \mathcal{X}\rightarrow \Delta, \mathcal{L})$ be a
 degeneration of polarized Calabi-Yau manifolds with a Calabi-Yau  variety  $X_{0}$ as the central fiber.   If $\omega_{t}$, $t\in \Delta^{*}$, is  the unique  Ricci-flat K\"{a}hler metric on $X_{t}$ with   $\omega_{t}\in c_{1}(\mathcal{L}|_{X_{t}})$, then $(X_{t}, \omega_{t}) $ converges to a compact metric space $X_{\infty}$ homeomorphic to $X_{0}$ in the Gromov-Hausdorff sense. As a consequence, the singular Ricci-flat K\"{a}hler  metric $\omega_{0}  \in  c_{1}(\mathcal{L}|_{X_{0}})$ induces a compact metric space structure on $X_{0}$.
\end{lemma}

\begin{proof}
We  denote $L_{t}=\mathcal{L}|_{X_{t}}$, and  denote $h_{t}$ the Hermitian metric on $L_{t}$, whose curvature  is the Ricci-flat K\"{a}hler-Einstein metric $\omega_{t}$.
 We apply Theorem \ref{D-S} to a sequence $t_{k} \rightarrow 0$, and then $X_{\infty}$ is a Calabi-Yau variety.
Furthermore, there are constants $m>0$ and $\bar{N}>0$ such that,  for any $k$,  there is an orthonormal basis $\Sigma_{t_{k}}$ of $H^{0}(X_{t_{k}}, L_{t_{k}}^{m})$ with respect to the $L^{2}$-norm  $\|\cdot\|_{L^{2}(h_{t_{k}}^{m})}$  inducing  an embedding $\Phi_{\Sigma_{k}}: X_{k} \hookrightarrow \mathbb{CP}^{\bar{N}}$ with $L_{t_{k}}^{m}= \Phi_{\Sigma_{k}}^{*} \mathcal{O}_{\mathbb{CP}^{\bar{N}}}(1)$. And $\Phi_{\Sigma_{t_{k}}}( X_{t_{k}})$ converges to $X_{\infty}$ in the Hilbert scheme $\mathcal{H}il_{\bar{N}}^{P_{m}}$ under  the natural  analytic topology.

Note that $\mathcal{L}^{m}$ is relative very ample, and thus
 there is a  morphism   $\Psi: \mathcal{X}\hookrightarrow \mathbb{CP}^{\bar{N}} \times\Delta \rightarrow  \mathbb{CP}^{\bar{N}}$  by composing an embedding and the projection   such that $\mathcal{L}^{m}=\Psi^{*}\mathcal{O}_{\mathbb{CP}^{\bar{N}}}(1)$.  Note that $ X_{t}$,  $t\in\Delta$, has the same Hilbert polynomial $P$, and hence  $P(m)=h^{0}(X_{t},L_{t}^{m})=\bar{N}+1$ for $m\gg1$.
 We  have that $\Psi( X_{t})$, for any $t\in\Delta$,  is   not included in any proper linear  subspace of $\mathbb{CP} ^{\bar{N}}$.
  If we denote  $\bar{h}_{FS}$   the Hermitian metric  on $\mathcal{O}_{\mathbb{CP}^{\bar{N}}}(1)$ whose curvature  is the Fubini-Study metric $\omega_{FS}$, then   the restriction $\bar{h}_{o,t}$ of $\Psi^{*}\bar{h}_{FS}$ on $X_{t}$ has the curvature $\bar{\omega}_{o,t}=\Psi^{*}\omega_{FS}|_{X_{t}}$.
  We can  choose an Hermitian metric  $h_{t}$ on $L_{t}^{m}$  by
  $h_{t}^{m}=e^{-m\varphi_{t}}\bar{h}_{o,t}$, where $\varphi_{t}$ is the potential function for  the unique Ricci-flat K\"ahler-Einstein metric in $c_{1}(L_{t})$, i.e. $\omega_{t}=\frac{1}{m} \bar{\omega}_{o,t}+ \sqrt{-1}\partial\overline{\partial}\varphi_{t}$ and $\sup\limits_{X_{t}} \varphi_{t}= 0$.

  Let $Z_{0}, \cdots, Z_{\bar{N}}$ be   sections of $\mathcal{O}_{\mathbb{CP}^{\bar{N}}}(1)$  such that the restrictions of $z_{0}=\Psi^{*}Z_{0}, \cdots, z_{\bar{N}}=\Psi^{*} Z_{\bar{N}}$ on $X_{t}$ form  a  basis  of $ H^{0}(X_{t}, L_{t}^{m})$, and  $z_{0}|_{X_{0}}, \cdots, z_{\bar{N}}|_{X_{0}}$   are orthonormal with respect to   $\|\cdot\|_{L^{2}(h_{0}^{m})}$.
 Now  for any compact  subset  $U\subset X_{0,reg}$,   by (\ref{metric convergence}), we have  $$\int_{F_{t}(U)}\langle z_{i}, z_{j}\rangle_{h_{t}^{m}} \omega_{t}^{n} \rightarrow \int_{U}\langle z_{i}, z_{j}\rangle_{h_{0}^{m}} \omega_{0}^{n}, \ \ {\rm when} \ t\rightarrow 0, \ \ {\rm and} $$ \[\begin{split}
  \int_{X_{t}\backslash F_{t}(U)}|\langle z_{i}, z_{j}\rangle_{h_{t}^{m}}| \omega_{t}^{n}  & \leq e^{mC}\sup_{X_{t}}|z_{i}|_{\bar{h}_{o,t}}|z_{j}|_{\bar{h}_{o,t}}{\rm Vol}_{\omega_{t}}(X_{t}\backslash F_{t}(U)) \\ & \leq C_{1} {\rm Vol}_{\omega_{t}}(X_{t}\backslash F_{t}(U)),
\end{split}\]
  for a constant $C_{1}>0$ independent of $t$.  Since ${\rm Vol}_{\omega_{t}}(X_{t})= {\rm Vol}_{\omega_{0}}(X_{0})={\rm Vol}_{\omega_{0}}(X_{0,reg})$,  we obtain $$ \int_{X_{t}}\langle z_{i}, z_{j}\rangle_{h_{t}^{m}} \omega_{t}^{n} \rightarrow \int_{X_{0}}\langle z_{i}, z_{j}\rangle_{h_{0}^{m}} \omega_{0}^{n}=\delta_{ij}$$ when $t\rightarrow0$,  by taking $U$ larger and larger, and a diagonal sequence. Thus there is a family of matrices    $v_{t}=(v_{t,ij})\in GL(\bar{N}+1)$ such that  $\Sigma_{t}'=\{\sum\limits_{i=0}^{\bar{N}}z_{i}v_{t,i0}, \cdots, \sum\limits_{i=0}^{\bar{N}}z_{i}v_{t,i\bar{N}}\}$ is an orthonormal basis of $ H^{0}(X_{t}, L_{t}^{m})$ with with respect to $\|\cdot\|_{L^{2}(h_{t}^{m})}$, and $v_{t}\rightarrow $Id  in $GL(\bar{N}+1)$ when $t\rightarrow 0$.

  There are    $u_{k}=(u_{k,ij})\in U(\bar{N}+1)$ such that $\Sigma_{t_{k}}=\Sigma_{t_{k}}'\cdot u_{k}$, and, by passing to a subsequence,  $u_{k} \rightarrow u_{0}$ in $U(\bar{N}+1)$.  Thus $w_{k}=(\det( v_{t_{k}}\cdot u_{k}))^{-\frac{1}{\bar{N}+1}} v_{t_{k}}\cdot u_{k}\rightarrow w_{0}=(\det(u_{0}))^{-\frac{1}{\bar{N}+1}}u_{0}$ in $SL(\bar{N}+1)$ when $t_{k}\rightarrow 0$. Under the $SL(\bar{N}+1)$-action $ \sigma$, we have $\Phi_{\Sigma_{t_{k}}}( X_{t_{k}})=\sigma ( w_{k}, \Psi(X_{t_{k}}))$, and $ X_{\infty}=\sigma (w_{0}, \Psi(X_{0}))$.
We proved      the conclusion that   $X_{\infty}$ is isomorphic to $X_{0}$.  By   \cite{RZ},    $(X_{t}, \omega_{t})$ converges to the compact metric space $X_{\infty}$ homeomorphic to $X_{0}$ in the Gromov-Hausdorff topology.
\end{proof}

  One corollary of this lemma is the uniqueness of the filling-in  for degenerations of Calabi-Yau manifolds.

 \begin{corollary}\label{coro}  Let  $(\pi_{\Delta}: \mathcal{X}\rightarrow \Delta, \mathcal{L})$ and $(\pi_{\Delta}': \mathcal{X}'\rightarrow \Delta, \mathcal{L}')$ be  two degenerations of polarized Calabi-Yau manifolds with  Calabi-Yau  varieties   $X_{0}$ and $X_{0}'$ as the central fibers respectively.  If  there is a sequence of points  $t_{k} \rightarrow 0$ in $\Delta$, and there is a sequence of isomorphism $\psi_{k}: X_{t_{k}}\rightarrow X_{t_{k}}'$ such that $\psi_{k}^{*}\mathcal{L}|_{X_{t_{k}}}\cong  \mathcal{L}'|_{X_{t_{k}}'}$, then $X_{0}$ is isomorphic to $X_{0}'$.
\end{corollary}

\begin{proof}  If $\omega_{t_{k}}$ and $\omega_{t_{k}}'$ are Ricci-flat K\"{a}hler-Einstein   metrics representing $c_{1}(\mathcal{L}|_{X_{t_{k}}})$ and $c_{1}(  \mathcal{L}'|_{X_{t_{k}}'})$  respectively, then $(X_{t_{k}}, \omega_{t_{k}})$ is isometric to $(X_{t_{k}}', \omega_{t_{k}}')$ as compact metric spaces. By Lemma \ref{pro0.1},  $(X_{t_{k}}, \omega_{t_{k}})$ (also $(X_{t_{k}}', \omega_{t_{k}}')$) converges to a compact metric space $X_{\infty}$ homeomorphic to both $X_{0}$ and $X_{0}'$.   We further claim that $X_{0}$ is isomorphic to $X_{0}'$.

By  the proof of Lemma \ref{pro0.1} and taking some powers  of $\mathcal{L}$ and $\mathcal{L}'$,  we can assume that there are   morphisms   $\Psi: \mathcal{X}\rightarrow  \mathbb{CP}^{N}$ and $\Psi': \mathcal{X}'\rightarrow  \mathbb{CP}^{N}$ with $\mathcal{L}=\Psi^{*}\mathcal{O}_{\mathbb{CP}^{N}}(1)$ and $\mathcal{L}'=\Psi'^{*}\mathcal{O}_{\mathbb{CP}^{N}}(1)$ respectively such that $\Psi|_{X_{t}}$ and $\Psi'|_{X_{t}'}$  are embeddings for any $t\in\Delta$. Furthermore,  there are embeddings $\Phi_{\Sigma_{k}}:X_{t_{k}}\rightarrow \mathbb{CP}^{N}$ induced by an  orthonormal basis $\Sigma_{k}$  of $H^{0}(X_{t_{k}}, \mathcal{L}|_{X_{t_{k}}})$ for each $k$ such that $\Phi_{\Sigma_{k}}(X_{t_{k}})$ converges to a Calabi-Yau variety homeomorphic to $X_{\infty}$ in the Hilbert scheme $\mathcal{H}il_{N}^{P}$, denoted still by $X_{\infty}$. The same  arguments as  in  the proof of Lemma \ref{pro0.1} show that there are $w$ and $w' \in SL(N+1)$ such that $ X_{\infty}=\sigma (w, \Psi(X_{0}))$ and $ X_{\infty}=\sigma (w', \Psi'(X_{0}'))$ where $\sigma:  SL(N+1)\times \mathcal{H}il_{N}^{P} \rightarrow \mathcal{H}il_{N}^{P}$ is a natural $SL(N+1)$-action on $ \mathcal{H}il_{N}^{P}$.  Hence $X_{0}$ is isomorphic to $X_{0}'$.
\end{proof}

\begin{remark} Note that $\mathcal{X}$ may not be birational to $\mathcal{X}'$  in this corollary.   If we have a stronger assumption that $\mathcal{X}\backslash X_{0}$ is isomorphic to $\mathcal{X}'\backslash X_{0}'$, then the conclusion   is a direct  consequence of   \cite[Theorem 2.1]{Bo} and  \cite[Corollary 4.3]{Od}.
\end{remark}

 \section{Quasi-projective moduli space}
 In this section, we construct an enlarged moduli space parameterizing   certain polarized Calabi-Yau varieties.
 \subsection{Moduli space for Calabi-Yau manifolds}
 In \cite{Vie}, Viehweg constructed   the coarse moduli space of polarized Calabi-Yau manifolds with a fixed  Hilbert polynomial $P$  by using  Geometric Invariant Theory (GIT), and it was shown to be a quasi-projective variety. Let's recall the relevant  notions and  the basic steps of the construction.

   The moduli functor  $\mathfrak{M}^{P}$ for polarized Calabi-Yau manifolds with Hilbert polynomial $P$   is a functor from the category of schemes to the category of sets such that $\mathfrak{M}^{P}({\rm Spec}(\mathbb{C}))=\mathcal{M}^{P}$ as a set, and for any scheme $T$, $\mathfrak{M}^{P}(T)= \{(\pi_{T}:\mathcal{X}\rightarrow T, \mathcal{L})\}/\sim $. Here $\pi_{T}:\mathcal{X}\rightarrow T$ is a flat family of schemes, and $\mathcal{L}$ is a  relative ample line bundle on $\mathcal{X}$ such that for any  point $t\in T$, $(X_{t}=\pi_{T}^{-1}(t), \mathcal{L}|_{X_{t}})\in \mathcal{M}^{P}$. And we say $(\pi_{T}:\mathcal{X}\rightarrow T, \mathcal{L})\sim (\pi_{T}':\mathcal{X}'\rightarrow T, \mathcal{L}')$ if there is a $T$-isomorphism $\tau: \mathcal{X} \rightarrow \mathcal{X}'$ and an  invertible sheaf $\mathcal{B}$ on $T$  such that $\tau^{*}\mathcal{L}'\cong \mathcal{L}\otimes \pi_{T}^{*}\mathcal{B}$.  Theorem 1.13 and Corollary 7.22  of  \cite{Vie} assert   that there is  a quasi-projective scheme $\tilde{\mathcal{M}}^{P}$ coarsely representing  the functor
     $\mathfrak{M}^{P} $, i.e. the following   hold.
 There is a natural transformation $\Theta: \mathfrak{M}^{P} \rightarrow \hom ( \cdot, \tilde{\mathcal{M}}^{P})$ such that $\Theta({\rm Spec}(\mathbb{C})): \mathfrak{M}^{P} ({\rm Spec}(\mathbb{C}))\rightarrow \hom ( {\rm Spec}(\mathbb{C}), \tilde{\mathcal{M}}^{P})$ is bijective, and,  for any scheme $W$ and  a natural transformation $\Xi: \mathfrak{M}^{P} \rightarrow \hom ( \cdot, W)$, there is a unique natural transformation $\Pi: \hom (\cdot,\tilde{\mathcal{M}}^{P}) \rightarrow \hom ( \cdot, W)$ such that $\Xi=\Pi\circ \Theta$. This property implies that for any $(\mathcal{X}  \rightarrow T, \mathcal{L})\in \mathfrak{M}^{P}(T)$, there is a unique morphism $T \rightarrow \tilde{\mathcal{M}}^{P}$.  We identify $\mathcal{M}^{P}$ with   the set of closed points of $\tilde{\mathcal{M}}^{P}$ by $\Theta({\rm Spec}(\mathbb{C}))$.

Now we recall the construction of  $\tilde{\mathcal{M}}^{P}$ in \cite{Vie}.  Firstly,   the functor $\mathfrak{M}^{P}$ is bounded. More precisely, by Matsusaka's Big Theorem (cf. \cite{Mat}),    for any polarized Calabi-Yau manifold $(X,L)$,   there is an  $m_{0}>0$ depending only on $P$    such that  for any $m\geqslant m_{0}$,  $L^{m}$ is very ample, and $H^{i}(X,L^{m})=\{0\}$, $i>0$.   By choosing a basis  $\Sigma$ of $H^{0}(X,L^{m_{0}})$,  we have an embedding $\Phi_{\Sigma}: X
 \hookrightarrow \mathbb{CP}^{N}$ such that  $L^{m_{0}}=\Phi_{\Sigma}^{*}\mathcal{O}_{\mathbb{CP}^{N}}(1)$. We   regard $\Phi_{\Sigma}(X)$ as a   point in the Hilbert scheme  $\mathcal{H}il_{N}^{P_{m_{0}}}$ parametrizing the subshemes of $\mathbb{CP}^{N}$ with   Hilbert polynomial $P_{m_{0}}$, where $N=h^{0} (X,L^{m_{0}})-1$.  For any other choice $\Sigma'$,   $\Phi_{\Sigma'}(X)=\sigma (u, \Phi_{\Sigma}(X))$ for a $u\in SL(N+1)$ where $\sigma : SL(N+1) \times \mathcal{H}il_{N}^{P_{m_{0}}}  \rightarrow  \mathcal{H}il_{N}^{P_{m_{0}}}$ is the  $SL(N+1)$-action on $\mathcal{H}il_{N}^{P_{m_{0}}}$ induced by the natural $SL(N+1)$-action on $\mathbb{CP}^{N}$.

 Secondly,   $\mathfrak{M}^{P}$ is open (see  \cite[Lemma 1.18]{Vie}), i.e. for any flat family of polarized varieties $(\pi_{Y}: \mathcal{X}  \rightarrow Y, \mathcal{L})$,  there is an  open subscheme $Y'\subset Y$ such that a morphism $T  \rightarrow Y$  factors through $T  \rightarrow Y'  \rightarrow Y$ if and only if $(\mathcal{X}\times_{Y}T  \rightarrow T,  \mathfrak{p}^{*}  \mathcal{L})\in \mathfrak{M}^{P} (T)$ where $\mathfrak{p}: \mathcal{X}\times_{Y}T  \rightarrow \mathcal{X}$ denotes the projection.  This is equivalent to    that there is an open  subscheme $\mathcal{H}_{N}^{o}$ of  $\mathcal{H}il_{N}^{P_{m_{0}}}$ (cf.  \cite[Notions 7.2]{Vie})  such that a point $p\in \mathcal{H}_{N}^{o}$ if and only if $(X_{p}=\pi_{\mathcal{H}}^{-1}(p), L_{p})\in \mathcal{M}^{P}$  and  $L_{p}^{m_{0}}\cong \mathcal{O}_{\mathbb{CP}^{N}}(1)|_{X_{p}}$ where $\pi_{\mathcal{H}}: \mathcal{U}_{N}  \rightarrow \mathcal{H}il_{N}^{P_{m_{0}}}$ is the universal family over the Hilbert scheme $\mathcal{H}il_{N}^{P_{m_{0}}}$.     The $SL(N+1)$-action $\sigma$  on $\mathcal{H}il_{N}^{P_{m_{0}}}$ induces an $SL(N+1)$-action  on $\mathcal{H}_{N}^{o}$, denoted still by $\sigma:  SL(N+1) \times \mathcal{H}_{N}^{o}  \rightarrow \mathcal{H}_{N}^{o}$.     The moduli scheme $\tilde{\mathcal{M}}^{P}$ is constructed by showing that a certain  quotient of the $SL(N+1)$-action   on $\mathcal{H}_{N}^{o}$ exists.

 Thirdly,   $\mathfrak{M}^{P}$  is separated (cf.  \cite[Lemma 1.18]{Vie}), i.e.  for any two $( \mathcal{X}_{i}\rightarrow S, \mathcal{L}_{i}) \in \mathfrak{M}^{P}(S)$, $i=1,2$,  any isomorphism of $(\mathcal{X}_{1}, \mathcal{L}_{1})$ onto $(\mathcal{X}_{2}, \mathcal{L}_{2})$ over $S\backslash \{0\}$ extends to a  $S$-isomorphism from  $(\mathcal{X}_{1}, \mathcal{L}_{1})$ to $(\mathcal{X}_{2}, \mathcal{L}_{2})$, where $(S,0)$ is a germ of smooth curve.  Thus the $SL(N+1)$-action  on $\mathcal{H}_{N}^{o}$ is proper and  the stabilizers are finite by  \cite[Lemma 7.3]{Vie}.
    The separatedness  condition implies that if the moduli space exists, i.e. the  quotient of the $SL(N+1)$-action $\sigma$ exists, then it is Hausdorff under the analytic topology.

 Finally,     $\mathfrak{M}^{P}$ satisfies further properties, so called the weak positivity and the weak stability, i.e. Assumption 7.19 of \cite{Vie} holds by  the proof of Theorem 1.13 in \cite{Vie}.  Then the geometric quotient $\mathcal{Q}^{o}: \mathcal{H}_{N}^{o}  \rightarrow \tilde{\mathcal{M}}^{P}$ of the $SL(N+1)$-action $\sigma$   exists (cf. Corollary 3.33 and the proof of Theorem 7.20 in \cite{Vie}). Here the geometric quotient means that
   $\mathcal{Q}^{o}$ is a morphism from $\mathcal{H}_{N}^{o}  $ to a scheme $  \tilde{\mathcal{M}}^{P}$ satisfying the following (cf.   \cite[Definition 3.6]{Vie}).
 For any $p\in  \mathcal{H}_{N}^{o}$ and any  $u\in SL(N+1)$,  $\mathcal{Q}^{o}(\sigma (u, p))= \mathcal{Q}^{o}(p)$,  $\mathcal{O}_{ \tilde{\mathcal{M}}^{P}}=(\mathcal{Q}^{o}_{*}\mathcal{O}_{\mathcal{H}_{N}^{o}})^{SL(N+1)}$,  and, for any two disjoint  $SL(N+1)$-invariant closed subscheme $W_{1}$ and $W_{2}$, we have that $\mathcal{Q}^{o}(W_{1})\bigcap \mathcal{Q}^{o}(W_{2})= \emptyset $, and both of $\mathcal{Q}^{o}(W_{1})$ and $\mathcal{Q}^{o} (W_{2})$ are closed.    Furthermore, for any point $p\in \tilde{\mathcal{M}}^{P}$,  the fiber $\mathcal{Q}^{o,-1}(p)$ consists of exactly one $SL(N+1)$-orbit.  The existence of $\mathcal{Q}^{o}$ implies that the $SL(N+1)$-action $\sigma$ is closed, and the dimension of the stabilizers is constant on connected components.  Moreover, $ \tilde{\mathcal{M}}^{P}$ is a quasi-projective variety, and   there is an ample sheaf $\lambda$ on $ \tilde{\mathcal{M}}^{P}$ such that  $\mathcal{Q}^{o,*}\lambda= \pi_{\mathcal{H},*} \varpi^{\nu}_{\mathcal{U}_{N} / \mathcal{H}_{N}^{o}}$ for a   $\nu\geqslant1$ by     \cite[Corollary 7.22]{Vie}.

 Viehweg also showed in Section 8 of   \cite{Vie} that the above construction works for more general moduli functor of  certain  polarized varieties with semi-ample dualizing sheaf and  at worst  canonical singularities as long as the boundedness condition, the openness (or the  local closedness) condition   and the  separatedness condition  hold.  In Section 3.2,   we will use  the construction in    \cite[Section 8]{Vie} to obtain an enlarged moduli space of $\mathcal{M}^{P} $, which also parameterizes certain Calabi-Yau varieties.

We remark that there is an analogue construction  of the symplectic reduction  (cf. \cite[Section 8]{M-GIT}) to  obtain $\mathcal{M}^{P} $ by using Ricci-flat K\"{a}hler-Einstein metrics.
If we define a  real slice \begin{equation}\label{e03.1} \mathcal{R}_{N}^{o}= \bigcup_{p\in\mathcal{H}_{N}^{o}} RO(X_{p},\mathcal{O}_{\mathbb{CP}^{N}}(1)|_{X_{p}})\subset \mathcal{H}_{N,red}^{o}, \end{equation}  where $ RO(X_{p},\mathcal{O}_{\mathbb{CP}^{N}}(1)|_{X_{p}})$ is defined by (\ref{eq2.2}), then there is a natural $SU(N+1)$-action on $\mathcal{R}_{N}^{o}$, and the set theory quotient is $\mathcal{M}^{P} $.  The real slice $ \mathcal{R}_{N}^{o}$ is an analog    of the zero level set of a  momentum  map in the symplectic reduction.

\subsection{Enlarged moduli space }
 Now we construct the enlarged moduli space.
 For any polarized Calabi-Yau  manifold  $(X,L)$ of dimension $n$ with  Hilbert polynomial $P=P(\mu)=\chi (X,L^{\mu})$, we assume  that $L$ is very ample, and $H^{i}(X,L^{\mu})=\{0\}$ for any $i>0$ and $\mu\geq1$ without loss of  generality.

   For any $m\geqslant 1$, there is an  $N=N(m)>0$ such that  a basis $\Sigma$ of $H^{0}(X, L^{m})$ induces an embedding $\Phi_{\Sigma}:X \hookrightarrow  \mathbb{CP}^{N}$.   Let $\pi_{\mathcal{H}}: \mathcal{U}_{N}\rightarrow \mathcal{H}ilb^{P_{m}}_{N}$ be the universal family over the Hilbert scheme $ \mathcal{H}ilb^{P_{m}}_{N}$ of the  Hilbert polynomial $P_{m}(\mu)=P(m\mu)$, and $\mathcal{H}_{N}^{o}\subset \mathcal{H}ilb^{P_{m}}_{N}$ be the open  subscheme whose set of closed points   parameterizes   smooth varieties.    The moduli space  $\mathcal{M}^{P}$ is constructed in \cite{Vie} as the geometric quotient $\mathcal{Q}^{o}: \mathcal{H}_{N}^{o} \rightarrow \mathcal{M}^{P}$ under the   natural $SL(N+1)$-action $\sigma$  on $\mathcal{H}_{N}^{o}$  as explained above.

  \begin{lemma}\label{l3.1}  There is an open  subscheme
    $\mathcal{H}_{N}$ of $\mathcal{H}ilb^{P_{m}}_{N}$ such that  $\mathcal{H}_{N}^{o}  \subset \mathcal{H}_{N} \subset \overline{ \mathcal{H}_{N}^{o}}$  where $\overline{ \mathcal{H}_{N}^{o}}$ denotes
      the Zariski closure  of   $\mathcal{H}_{N}^{o}$ in $ \mathcal{H}ilb^{P_{m}}_{N}$, and  a  point   $p\in \mathcal{H}_{N}$ if and only if  $X_{p}=\pi_{\mathcal{H}}^{-1}(p)$ is a Calabi-Yau variety.
\end{lemma}

\begin{proof} This result is undoubtedly well-known to experts, and the proof was explained  to the author by Chenyang Xu.   We use the closed subscheme $ \overline{\mathcal{H}_{N}^{o}}$ to replace  $ \mathcal{H}ilb^{P_{m}}_{N}$, and try to  prove that $\mathcal{H}_{N}$ is an open subscheme of $ \overline{\mathcal{H}_{N}^{o}}$.

 Note that Calabi-Yau varieties are normal projective varieties, and
the normality is an open condition for flat families. Any subset $E$ of $ \overline{\mathcal{H}_{N}^{o}}$ containing the open subset   $  \mathcal{H}_{N}^{o}$ is constructible (cf. \cite[Proposition 10.14]{Ha}), since for any irreducible closed subset $Y$, if $Y$ is proper, $E\bigcap Y$ is nowhere dense, and otherwise  $Y=\overline{\mathcal{H}_{N}^{o}}$, which  contains $ \mathcal{H}_{N}^{o}$.    The main theorem in   \cite{Ka} shows that if $\pi_{S}: \mathcal{X}\rightarrow S$ is a flat morphism from a germ of a variety to a germ of smooth curve $(S, 0)$ whose special fiber $X_{0}=\pi_{S}^{-1}(0)$ has only canonical singularities, then both $\mathcal{X}$ and fibers $X_{t}$ have only canonical singularities.  Thus for any point $p\in \overline{\mathcal{H}_{N}^{o}}$, if $X_{p}=\pi_{\mathcal{H}}^{-1}(p)$ has at worst  canonical singularities, then  by taking curves passing $p$ and normalizations  of curves, there is a neighborhood of $p$ over which fibers of $\pi_{\mathcal{H}}$  have
 at worst  canonical singularities, i.e. having only canonical singularities is an open condition.
  We denote $\mathcal{W}$ the open subscheme of $\overline{\mathcal{H}_{N}^{o}}$ whose set of closed points    parameterizes normal varieties with at worst    canonical singularities.   By \cite[Lemma 1.19]{Vie},  there is a locally closed subscheme $\mathcal{H}_{N}$ of $\mathcal{W}$ such that a morphism $T\rightarrow \mathcal{W}$ factors through $T\rightarrow \mathcal{H}_{N}$, if and only if $( \mathcal{U}_{N}\times_{\mathcal{W}}T\rightarrow T, \mathfrak{p}^{*}\varpi_{ \mathcal{U}_{N}/\mathcal{W}})\sim  ( \mathcal{U}_{N}\times_{\mathcal{W}}T\rightarrow T, \mathfrak{p}^{*}\mathcal{O}_{ \mathcal{U}_{N}}) $, where $ \mathfrak{p}: \mathcal{U}_{N}\times_{\mathcal{W}}T \rightarrow \mathcal{U}_{N}$ is the projection. Hence  a point $p\in \mathcal{H}_{N}$ if and only if $\varpi_{X_{p}}\cong \mathcal{O}_{X_{p}}$, which implies that $X_{p}$ is a Calabi-Yau variety. Furthermore, $\mathcal{H}_{N}^{o}  \subset \mathcal{H}_{N}$, and $ \mathcal{H}_{N}$ is open.   \end{proof}

We define a moduli subfunctor $\mathfrak{M}_{m}$ of the moduli functor of polarized Gorenstein varieties, i.e. 3) of Examples 1.4 in \cite{Vie},   such that $$\mathfrak{M}_{m}({\rm Spec}(\mathbb{C}))=\{(X_{p}=\pi_{\mathcal{H}}^{-1}(p),\mathcal{O}_{\mathbb{CP}^{N}}(1)|_{X_{p}})|\  p \ {\rm  is \ a \    point  \ of}   \ \mathcal{H}_{N} \}/\sim,  $$ where $(X_{p_{1}},\mathcal{O}_{\mathbb{CP}^{N}}(1)|_{X_{p_{1}}})\sim (X_{p_{2}},\mathcal{O}_{\mathbb{CP}^{N}}(1)|_{X_{p_{2}}})$ if and only if   there is an isomorphism $\psi: X_{p_{1}}  \rightarrow X_{p_{2}}$ such that $\mathcal{O}_{\mathbb{CP}^{N}}(1)|_{X_{p_{1}}}\cong \psi^{*} \mathcal{O}_{\mathbb{CP}^{N}}(1)|_{X_{p_{2}}}$, which is equivalent  to $p_{1}=\sigma(u, p_{2})$ for an  $u\in SL(N+1)$. The functor $\mathfrak{M}_{m}$ is bounded by the definition, and is  open  by   Lemma \ref{l3.1}.
  By  \cite[Corollary 4.3]{Od} or the unpublished work  \cite[Theorem 2.1]{Bo},  if   $( \mathcal{X}_{1}\rightarrow S, \mathcal{L}_{1})$ and $( \mathcal{X}_{2}\rightarrow S, \mathcal{L}_{2})$ are     two flat families of polarized  Calabi-Yau varieties  over a germ of smooth curve  $(S,0)$, then any isomorphism of these two families over $S\backslash\{0\}$ extends to an isomorphism over $S$, i.e.  the moduli  functor  $\mathfrak{M}_{m}$ is also  separated.

Now we  use the construction in    \cite[Section 8]{Vie} to prove that $\mathfrak{M}_{m}$ can be coarsely represented by a quasi-projective variety.

   \begin{lemma}\label{t3.2}   The coarse moduli space of $\mathfrak{M}_{m}$ is a quasi-projective  variety      $  \tilde{\mathcal{M}}_{m}$, which is constructed as  a geometric quotient $\mathcal{Q}: \mathcal{H}_{N} \rightarrow \tilde{\mathcal{M}}_{m}$. There is a positive  integer $\nu=\nu(m)$, and an ample line bundle $\lambda_{m}$ on  $\tilde{\mathcal{M}}_{m} $ such that $\mathcal{Q}^{*}\lambda_{m}=\pi_{\mathcal{H},*} \varpi^{\nu}_{\mathcal{U}_{N}/ \mathcal{H}_{N}}$.
     Furthermore,   $\tilde{\mathcal{M}}^{P}$ is an  open  subscheme  of $ \tilde{\mathcal{M}}_{m}$.
\end{lemma}

\begin{proof} Note that  $\mathfrak{M}_{m}$ satisfies    \cite[Assumptions 8.22]{Vie}, i.e. $\mathfrak{M}_{m}$ is bounded, open, separated, and moreover  is a moduli  functor of varieties with semi-ample dualizing sheaf.  Then  \cite[Theorem 8.23]{Vie} shows that $\mathfrak{M}_{m}$ can be coarsely represented by a quasi-projective scheme $ \tilde{\mathcal{M}}_{m}$. More precisely, the base  changing,  local freeness condition, the weak positivity and the weak stability are verified in    \cite[Section 8.6]{Vie}, and then  \cite[ Theorem 7.20]{Vie} shows the existence of the geometric quotient $\mathcal{Q}: \mathcal{H}_{N}\rightarrow \tilde{\mathcal{M}}_{m}$.  Furthermore, there is a positive  integer $\nu=\nu(m)$, and an ample line bundle $\lambda_{m}$ on  $\tilde{\mathcal{M}}_{m} $ such that $\mathcal{Q}^{*}\lambda_{m}=\pi_{\mathcal{H},*} \varpi^{\nu}_{\mathcal{U}_{N}/ \mathcal{H}_{N}}$ by  \cite[Theorem 8.23, Theorem 7.20 and Corollary 7.22]{Vie}.
Finally, since  $\mathcal{H}_{N}^{o}\subset \mathcal{H}_{N}$ is $SL(N+1)$-invariant Zariski open,   and $\mathcal{Q}|_{\mathcal{H}_{N}^{o}}=\mathcal{Q}^{o}$, we obtain that  $\tilde{\mathcal{M}}^{P}$ is  open  in  $ \tilde{\mathcal{M}}_{m}$.
\end{proof}

\begin{remark}  For a $p\in \mathcal{H}_{N}$, if $X_{p}$ is smooth, then $\mathcal{O}_{\mathbb{CP}^{N}}(1)|_{X_{p}}=L^{m}$ where $L$ is an   ample line bundle on $X_{p}$, and however, if $X_{p}$ is singular, there may not exist such ample line bundle.  Thus some Calabi-Yau variety here could   not be embedded to  the lower dimensional projective space.
\end{remark}

We also have an enlarged real slice of $ \mathcal{R}_{N}^{o}$.
We define   \begin{equation}\label{e3.1} \mathcal{R}_{N}= \bigcup_{p\in\mathcal{H}_{N}} RO(X_{p},\mathcal{O}_{\mathbb{CP}^{N}}(1)|_{X_{p}})\subset \mathcal{H}_{N,red},\end{equation} where $\mathcal{H}_{N,red}$ is the reduced variety of $\mathcal{H}_{N}$ with the natural analytic topology, and  $RO(X_{p},\mathcal{O}_{\mathbb{CP}^{N}}(1)|_{X_{p}})$ is the $SU(N+1)$-orbit induced by the Ricci-flat
 K\"{a}hler-Einstein  metric $\omega\in c_{1}(\mathcal{O}_{\mathbb{CP}^{N}}(1)|_{X_{p}})$. By (\ref{eq2.2}),    $  RO(X_{p},\mathcal{O}_{\mathbb{CP}^{N}}(1)|_{X_{p}})\subset O(X_{p},\mathcal{O}_{\mathbb{CP}^{N}}(1)|_{X_{p}})\bigcap \mathcal{R}_{N}$, and  by the uniqueness of the K\"ahler-Einstien metric  $\omega$, we obtain
 $$ RO(X_{p},\mathcal{O}_{\mathbb{CP}^{N}}(1)|_{X_{p}})= O(X_{p},\mathcal{O}_{\mathbb{CP}^{N}}(1)|_{X_{p}})\bigcap \mathcal{R}_{N}, $$ for any $p\in\mathcal{H}_{N}$.
 The set theory  quotient space   \begin{equation}\label{e3.2}  \mathcal{M}_{m}=  \mathcal{R}_{N}/ SU(N+1)= \mathcal{H}_{N,red}/ SL(N+1)\end{equation} with the quotient topology induced by the analytic topology of $\mathcal{H}_{N,red}$,  is homeomorphic to the underlying variety of $ \tilde{\mathcal{M}}_{m}$. Note that the reduced Hilbert scheme $ \mathcal{H}ilb^{P_{m}}_{N,red}$ is Hausdorff, and so is   the subset  $\mathcal{R}_{N}$. Thus the quotient by a  compact Lie group  $\mathcal{M}_{m}=  \mathcal{R}_{N}/ SU(N+1)$ is also Hausdorff, which has already been  implied by the separatedness of $\mathfrak{M}_{m}$.  For a point  $p\in\mathcal{H}_{N}$, we denote $[X_{p}]\in \mathcal{M}_{m}$ the  image of $p$ under  the quotient map, i.e. $[X_{p}]=\mathcal{Q}(p)$.

 \section{Proof of Theorem  \ref{main} and Theorem  \ref{main2}}
 In this section, we prove Theorem  \ref{main} and Theorem  \ref{main2}. By  Matsusaka's Big Theorem,    for any polarized Calabi-Yau manifold $(X,L)\in \mathcal{M}^{P}$,  we assume that  for any $\mu \geqslant 1$,  $L^{\mu}$ is very ample, and $H^{i}(X,L^{\mu})=\{0\}$, $i>0$ without loss generality.

For any $D>0$, we define a subset $\mathcal{M}^{P}(D)$  of  $\mathcal{M}^{P}$ by  $$\mathcal{M}^{P}(D)=\{ \left[X,L\right] \in \mathcal{M}^{P}| \  {\rm   Ricci-flat \  metric} \   \omega \in c_{1}(L)  \ {\rm   with   \  diam}_{\omega}(X)\leq D \}. $$
 We have that if $D_{1}\leq D_{2}$, then $\mathcal{M}^{P}(D_{1})\subset \mathcal{M}^{P}(D_{2})$, and $\mathcal{M}^{P}=\bigcup\limits_{D>0}\mathcal{M}^{P}(D) $.
 Let us  consider an exhaustion $$\mathcal{M}^{P}(1)\subset \cdots \subset \mathcal{M}^{P}(j)\subset \cdots \cdots \subset \mathcal{M}^{P}=\bigcup\limits_{j\in \mathbb{N}}\mathcal{M}^{P}(j).$$

  Note that for  a sequence  $\left[X_{k}, L_{k}\right]\in \mathcal{M}^{P}(j)$, if $(X_{k}, \omega_{k})$   converges to a compact metric space $X_{\infty}$ in the Gromov-Hausdorff sense, then by Theorem \ref{D-S}, there are embeddings   $\Phi_{k}: X_{k} \hookrightarrow \mathbb{CP}^{N_{j}}$ for an  $N_{j}>0$ independent of $k$ such that  $L_{k}^{m_{j}}\cong \Phi_{k}^{*} \mathcal{O}_{\mathbb{CP}^{N_{j}}}(1)$ for an $m_{j}>0$, and
    $\Phi_{k}(X_{k})$ converges to
  a Calabi-Yau variety in the Hilbert scheme  $\mathcal{H}ilb^{P_{m_{j}}}_{N_{j}}$, which is  homeomorphic to $X_{\infty}$, denoted still by $X_{\infty}$.

   \begin{lemma}\label{l4.1}  If we   denote $m(l)=\prod\limits_{j=1}^{l} m_{j}$, and the sequence  $\left[X_{k}, L_{k}\right]\in \mathcal{M}^{P}(l_{0})$,  for an   $l_{0} \leqslant l$, i.e.  ${\rm diam}_{\omega_{k}}(X_{k})\leq l_{0} $,    then $\left[X_{\infty}\right]\in  \mathcal{M}_{m(l)}$, where $\mathcal{M}_{m(l)}$ is  the underlying  quasi-projective variety of $\tilde{\mathcal{M}}_{m(l)}$ constructed in Lemma \ref{t3.2}.
\end{lemma}

  \begin{proof}  Note that $X_{\infty}$ is a Calabi-Yau variety,  $ \mathcal{O}_{\mathbb{CP}^{N_{l_{0}}}}(1)|_{X_{\infty}} $ is very ample, and
    $(X_{\infty}, \mathcal{O}_{\mathbb{CP}^{N_{l_{0}}}}(1)|_{X_{\infty}}) $
   represents  a point in the Hilbert scheme $\mathcal{H}ilb^{P_{m_{l_{0}}}}_{N_{l_{0}}}$.  If we denote $m_{l,l_{0}}=m(l)/m_{l_{0}}$, then $ \mathcal{O}_{\mathbb{CP}^{N_{l_{0}}}}(m_{l,l_{0}})|_{X_{\infty}} $ is very ample and without any higher cohomology (cf.  \cite[Corollary 2.36]{Vie}).  The Hilbert polynomial of  $(X_{\infty}, \mathcal{O}_{\mathbb{CP}^{N_{l_{0}}}}(m_{l,l_{0}})|_{X_{\infty}}) $ is $P_{m(l)}(k)=P_{m_{l_{0}}}(m_{l,l_{0}}\cdot k)$.  Thus there is  an embedding $\Psi_{\infty}: X_{\infty} \rightarrow  \mathbb{CP}^{N_{m(l)}}$ such that $\mathcal{O}_{\mathbb{CP}^{N_{l_{0}}}}(m_{l,l_{0}})|_{X_{\infty}}=\Psi_{\infty}^{*}\mathcal{O}_{\mathbb{CP}^{N_{m(l)}}}(1)$, where $N_{m(l)}=P_{m(l)}(1)-1$.
  We have  $\Psi_{\infty}(X_{\infty})\in \mathcal{H}_{N_{m(l)}} \subset \mathcal{H}ilb^{P_{m(l)}}_{N_{m(l)}}$. Note that  $[\Psi_{\infty}(X_{\infty})]=\mathcal{Q}_{l}(\Psi_{\infty}(X_{\infty}))$, where  $\mathcal{Q}_{l}:  \mathcal{H}_{N_{m(l)}} \rightarrow \tilde{\mathcal{M}}_{m(l)}$ is the quotient map in Lemma \ref{t3.2}.
     We obtain the conclusion by identifying $\Psi_{\infty}(X_{\infty})$ with $X_{\infty}$.
  \end{proof}

\begin{lemma}\label{l01.1}  There is a continuous   inclusion $i_{l}:  \mathcal{M}_{m(l)}  \hookrightarrow  \mathcal{M}_{m(l+1)} $. \end{lemma}

\begin{proof} For any $l$,  let $\pi_{\mathcal{H}}:  \mathcal{U}_{N_{m(l)}}\rightarrow \mathcal{H}ilb^{P_{m(l)}}_{N_{m(l)}} $ be the universal family. The line bundle  $\mathcal{L}_{m(l)}=\Upsilon^{*}\mathcal{O}_{\mathbb{CP}^{N_{m(l)}}}(1)$ is a  relative very ample  where $\Upsilon$ is the composition of the  embedding and the projection $ \mathcal{U}_{N_{m(l)}}\hookrightarrow \mathbb{CP}^{N_{m(l)}}\times \mathcal{H}ilb^{P_{m(l)}}_{N_{m(l)}}\rightarrow \mathbb{CP}^{N_{m(l)}}$.  For a point  $p\in \mathcal{H}_{N_{m(l)}}$, if  $X_{p}=\pi_{\mathcal{H}}^{-1}(p)$,  then $\mathcal{L}_{m(l)}^{m_{l+1}}|_{X_{p}}$ has no higher cohomology, and $(X_{p},\mathcal{L}_{m(l)}^{m_{l+1}}|_{X_{p}})$ has  the Hilbert polynomial
$P_{m(l+1)}(k)=P_{m(l)}(m_{l+1}\cdot k)$. Note that $N_{m(l+1)}= P_{m(l+1)}(1)-1$.   Thus we have  embeddings  $\Theta: \pi_{\mathcal{H}}^{-1}(\mathcal{H}_{N_{m(l)}})\hookrightarrow \mathbb{P}( \pi_{\mathcal{H}*}(\mathcal{L}_{m(l)}^{m_{l+1}}))\cong \mathbb{CP}^{N_{m(l+1)}}\times \mathcal{H}_{N_{m(l)}}$ with $\mathcal{L}_{m(l)}^{m_{l+1}}=\Theta^{*}\mathcal{L}_{m(l+1)}$. For two choices of embeddings $\Theta$ and $\Theta'$, there is a section $u\in \mathcal{O}_{\mathcal{H}_{N_{m(l)}}}( SL(N_{m(l+1)}+1)\times \mathcal{H}_{N_{m(l)}})$ such that $\Theta'(X_{p})=\sigma(u(p), \Theta(X_{p}))$ under the $SL(N_{m(l+1)}+1)$-action $\sigma$  on $\mathcal{H}ilb^{P_{m(l+1)}}_{N_{m(l+1)}}$.    By the universal property of Hilbert scheme, we obtain a morphism $\mathfrak{I}_{l}: \mathcal{H}_{N_{m(l)}}\rightarrow \mathcal{H}ilb^{P_{m(l+1)}}_{N_{m(l+1)}}$ such that $\pi_{\mathcal{H}}^{-1}(\mathcal{H}_{N_{m(l)}})= \mathcal{U}_{N_{m(l+1)}}\times_{\mathcal{H}ilb^{P_{m(l+1)}}_{N_{m(l+1)}}} \mathcal{H}_{N_{m(l)}}$, and furthermore,  $\mathfrak{I}_{l}( \mathcal{H}_{N_{m(l)}})\subset \mathcal{H}_{N_{m(l+1)}}$. We also denote $\mathfrak{I}_{l}$ the corresponding morphism $\pi_{\mathcal{H}}^{-1}(\mathcal{H}_{N_{m(l)}}) \rightarrow  \mathcal{U}_{N_{m(l+1)}}$ without any confusion.    If $p$ and $p'\in \mathcal{H}_{N_{m(l)}}$ satisfy $p=\sigma (w, p')$ for a $w\in SL(N_{m(l)}+1)$, then  $(X_{p},\mathcal{L}_{m(l)}^{m_{l+1}}|_{X_{p}})\sim (X_{p'},\mathcal{L}_{m(l)}^{m_{l+1}}|_{X_{p'}})$.  Hence $\mathfrak{I}_{l}$  is equivariant  under the  $ SL(N_{m(l)}+1)$ and $ SL(N_{m(l+1)}+1)$ actions.  We obtain a map    $i_{l}:  \mathcal{M}_{m(l)}  \rightarrow  \mathcal{M}_{m(l+1)} $ by taking the quotients, which is continuous under the analytic topology.

If  $i_{l}([X_{1}])=i_{l}([X_{2}])$, for two $[X_{1}]$ and $[X_{2}]\in \mathcal{M}_{m(l)}$,   then there is an isomorphism $\tilde{\psi}: \mathfrak{I}_{l}(X_{1})  \rightarrow \mathfrak{I}_{l}(X_{2})$ with $\tilde{\psi}^{*} \mathcal{L}_{m(l+1)}|_{\mathfrak{I}_{l}(X_{2})}\cong \mathcal{L}_{m(l+1)}|_{\mathfrak{I}_{l}(X_{1})}$.
 Thus we obtain  an isomorphism $\psi: X_{1}  \rightarrow X_{2}$ with $\psi^{*} \mathcal{L}_{m(l)}^{m(l+1)}|_{X_{2}}\cong \mathcal{L}_{m(l)}^{m(l+1)}|_{X_{1}}$.     Hence $\psi^{*} \mathcal{L}_{m(l)}|_{X_{2}}\cong \mathcal{L}_{m(l)}|_{X_{1}}$, and $[X_{1}]=[X_{2}]$,  i.e.  $i_{l}$ is injective.
\end{proof}

  \begin{proof}[Proof of Theorem  \ref{main} ]  We define  \begin{equation}\overline{\mathcal{M}}^{P}= \bigcup_{l\in\mathbb{N}}\mathcal{M}_{m(l)}\end{equation} by  using the inclusions $i_{l}$.  Note that $\mathcal{M}^{P}$ is an open dense  subset of each $\mathcal{M}_{m(l)}$ by Lemma \ref{t3.2} and thus of  $\overline{\mathcal{M}}^{P}$.

   We extend  the map $\mathcal{CY}: \mathcal{M}^{P} \rightarrow \mathcal{M}et$ to a     map $\overline{\mathcal{CY}}: \overline{\mathcal{M}}^{P}\rightarrow \overline{\mathcal{CY}( \mathcal{M}^{P})}$ by the following. For any $x\in \mathcal{M}_{m(l)} \subset \overline{\mathcal{M}}^{P}$, let $\omega_{l}$ be the Ricci-flat K\"{a}hler-Einstein metric on $X_{p}$ representing $c_{1}(\mathcal{O}_{\mathbb{CP}^{N_{m(l)}}}|_{X_{p}})$, where $p\in \mathcal{H}_{N_{m(l)}}\subset \mathcal{H}ilb^{P_{m(l)}}_{N_{m(l)}}$,  $\mathcal{Q}_{m(l)}(p)=x$, and $X_{p}=\pi_{\mathcal{H}}^{-1}(p)$ from the construction in Section 3.2.  For a  normalized curve $f: \Delta \rightarrow \mathcal{H}_{N_{m(l)}}$ with $f(\Delta^{*} )\subset \mathcal{H}_{N_{m(l)}}^{o}$, we have a degeneration of polarized Calabi-Yau manifolds $(\mathcal{U}_{N_{m(l)}}\times_{\mathcal{H}_{N_{m(l)}}} \Delta \rightarrow \Delta,  \mathfrak{p}^{*}\mathcal{O}_{\mathbb{CP}^{N_{m(l)}}}(1)|_{\mathcal{U}_{N_{m(l)}}} ) $ with central fiber $X_{p}$, where $ \mathfrak{p}$ is the projection to the first factor.
     By Lemma \ref{pro0.1},  $\omega_{l}$ induces a compact metric space structure on $X_{p}$.

  We define   $$  \overline{\mathcal{CY}}(x)=(X_{p}, \frac{1}{m(l)}\omega_{l}).$$    If we consider $i_{l}(x)\in  \mathcal{M}_{m(l+1)}$, then $X_{p}$ is isomorphic to $\mathfrak{I}_{l}(X_{p})$,  $(X_{p}, \frac{1}{m(l)}\omega_{l})$ is isometric to $(\mathfrak{I}_{l}(X_{p}), \frac{1}{m(l+1)}\omega_{l+1})$ as metric spaces  by $i_{l}(x)=\mathcal{Q}_{m(l+1)} (\mathfrak{I}_{l}(X_{p}) ) $, and    $\mathfrak{I}_{l}^{*}\mathcal{O}_{\mathbb{CP}^{N_{m(l+1)}}}(1)|_{\mathfrak{I}_{l}(X_{p})}\cong \mathcal{O}_{\mathbb{CP}^{N_{m(l)}}}(m_{l+1})|_{X_{p}}$.  Thus     $\overline{\mathcal{CY}}$ is   well-defined.
  If $X_{p}$ is smooth, there is an ample line bundle $L_{p}$ such that $L_{p}^{m(l)}\cong \mathcal{O}_{\mathbb{CP}^{N_{m(l)}}}(1)|_{X_{p}}$ and $[X_{p}, L_{p}]\in \mathcal{M}^{P}$. Hence $\frac{1}{m(l)}\omega_{l} \in c_{1}(L_{p})$ and $\overline{\mathcal{CY}}|_{\mathcal{M}^{P}}=\mathcal{CY}$.

   For any compact metric space $(Y, d_{Y}) \in \overline{\mathcal{CY}(\mathcal{M}^{P})}$,  we have a sequence  $\left[X_{k}, L_{k}\right]\in \mathcal{M}^{P}(l_{1})$, for an $l_{1}> {\rm diam}_{d_{Y}}(Y)$,  such that  $(X_{k}, \omega_{k})$   converges to $(Y, d_{Y})$ in the Gromov-Hausdorff sense, where $\omega_{k}\in c_{1}(L_{k})$ is the Ricci-flat K\"{a}hler-Einstein metric.  By Lemma \ref{l4.1}, there is a Calabi-Yau variety $X_{\infty}$ homeomorphic to $Y$ and satisfying that  $X_{\infty}$ can be embedded in $ \mathbb{CP}^{N_{m(l_{1})}}$ and  $\left[X_{\infty}\right]\in  \mathcal{M}_{m(l_{1})}$. Furthermore, the metric structure $d_{Y}$ is induced by the singular Ricci-flat K\"{a}hler-Einstein metric $\omega \in \frac{1}{m(l_{1})} c_{1} (\mathcal{O}_{\mathbb{CP}^{N_{m(l_{1})}}}(1)|_{X_{\infty}})$ by Theorem \ref{D-S}, which implies that $\overline{\mathcal{CY}} (\left[X_{\infty}\right])=(X_{\infty}, \omega)$,  i.e.   $\overline{\mathcal{CY}} $ is  surjective.
      We obtain i), ii)  and iii).

    Let $(\pi_{\Delta}: \mathcal{X} \rightarrow \Delta, \mathcal{L}) $ be a degeneration of polarized  Calabi-Yau manifolds satisfying the condition in iv).  We assume that $\mathcal{L}$ is relative very ample, and $[X_{t}, \mathcal{L}|_{X_{t}}]\in \mathcal{M}^{P}$,  $t\in\Delta^{*}$. By (\ref{diam}), there is an $l_{2}>0$ such that ${\rm diam}_{\omega_{t}}(X_{t})\leq l_{2}$ for $t\in \Delta^{*}$, where $\omega_{t}$ is the unique Ricci-flat K\"{a}hler-Einstein metric representing $c_{1} (\mathcal{L}|_{X_{t}})$. By the above construction,  we have a $m(l_{2})>0$ such that we have a morphism  $\tilde{\Psi}:\mathcal{X}\hookrightarrow   \Delta\times \mathbb{CP}^{N_{m(l_{2})}}\rightarrow  \mathbb{CP}^{N_{m(l_{2})}}$ with  $ \mathcal{L}^{m(l_{2})}|_{X_{t}}\cong \tilde{\Psi}^{*} \mathcal{O}_{\mathbb{CP}^{N_{m(l_{2})}}}(1)|_{X_{t}}$.  There is a  unique morphism $\rho: \Delta \rightarrow \mathcal{M}_{m(l_{2})} $ such that $\rho(t)=[\tilde{\Psi}(X_{t})]$  by  the conclusion of $\mathcal{M}_{m(l_{2})}$ coarsely representing the functor  $\mathfrak{M}_{m(l_{2})}$ in Lemma  \ref{t3.2}.
   The Gromov-Hausdorff convergence in iv)  is a consequence of  Lemma \ref{pro0.1}.
    \end{proof}

  \begin{proof}[Proof of Theorem    \ref{main2}]  For any point  $x\in  \overline{\mathcal{M}}^{P}\backslash   \mathcal{M}^{P}$, we assume that $x\in  \mathcal{M}_{m(l)} \subset \overline{\mathcal{M}}^{P}$ for an $m(l)>0$.  Let $p\in \mathcal{H}_{N_{m(l)}}\subset \mathcal{H}ilb^{P_{m(l)}}_{N_{m(l)}}$,  $\mathcal{Q}_{m(l)}(p)=x$, and $X_{p}=\pi_{\mathcal{H}}^{-1}(p)$, where  $\pi_{\mathcal{H}}: \mathcal{U}_{N_{m(l)}}\rightarrow  \mathcal{H}_{N_{m(l)}}$ is the universal family.  Let $\tau:\Delta\rightarrow \mathcal{H}_{N_{m(l)}}$ be a morphism such that $\tau(0)=p$ and $\tau(\Delta^{*})\subset  \mathcal{H}_{N_{m(l)}}^{o}$, and $\mathcal{X}=\mathcal{U}_{N_{m(l)}}\times_{\mathcal{H}_{N_{m(l)}}}\Delta\rightarrow \Delta$ be the  degeneration of Calabi-Yau manifolds. Since the central fiber $X_{p}$ is a Calabi-Yau variety, \cite[Proposition 2.3]{Wang1} shows that the Weil-Petersson distance between the interior $\Delta^{*}$ and $p$ is finite.  Hence we obtain i) by composing the quotient map $\mathcal{Q}_{m(l)}$.

     Let $(\pi_{\Delta}: \mathcal{X} \rightarrow \Delta, \mathcal{L})$ be a degeneration of polarized  Calabi-Yau manifolds.    If we assume that the Weil-Petersson distance between  $\Delta^{*}$ and $0$ is finite, then $(\pi_{\Delta}: \mathcal{X} \rightarrow \Delta, \mathcal{L})$ is birational to a new family $(\pi_{\Delta}': \mathcal{X}' \rightarrow \Delta, \mathcal{L}')$  such that $(\mathcal{X}\backslash X_{0}, \mathcal{L})\cong (\mathcal{X}'\backslash X_{0}', \mathcal{L}')$, and $X_{0}'$ is a Calabi-Yau variety  by   \cite[Theorem 1.2]{To}. We obtain  ii) by  the same argument as in the proof of iv) in  Theorem  \ref{main}.
\end{proof}

 \begin{remark}  Note that   $\overline{\mathcal{M}}^{P}$  parameterizes certain Calabi-Yau varieties, which are proven to be K-stable by \cite{Od0}.    Hence Theorem \ref{main}  gives an evidence to the conjecture of the existence of   K-moduli spaces  (cf. \cite[Conjecture 3.1]{Od}).
\end{remark}

 \begin{remark} In the case of K3 surfaces, Theorem \ref{main} is weaker than  results of  \cite{Kob,KT}. Here the moduli space that we consider  has more restrictions  than those moduli spaces in \cite{Kob,KT}, which parameterize  K3 surfaces with generalized  (not necessary integral) polarizations and without any fixed  Hilbert polynomial. 
\end{remark}

\section{A remark for compactifications}
Finally,   we remark that there actually  is   a natural Gromov-Hausdorff compactification of  $ \overline{\mathcal{M}}^{P}$.    If we define the  normalized Calabi-Yau map  $$ \mathcal{NCY}: \mathcal{M}^{P}  \rightarrow \mathcal{M}et,  \  \  {\rm by}  \  \  \left[X,L\right]  \mapsto  (X, {\rm diam}_{\omega}^{-2}(X)\omega),$$
 where $\omega\in c_{1}(L)$ is the unique Ricci-flat K\"ahler-Einstein metric,  then the  Gromov's precompactness theorem (cf. \cite{Grom,Rong})  asserts that   the closure $\overline{ \mathcal{NCY}( \mathcal{M}^{P} )} $  of $ \mathcal{NCY}( \mathcal{M}^{P} )$ in $ \mathcal{M}et$ is compact.    Moreover,  the map $$ \overline{\mathcal{CY}(\mathcal{M}^{P})}  \rightarrow  \overline{ \mathcal{NCY}( \mathcal{M}^{P} )},  \  \  \  (Y , d_{Y}) \mapsto (Y,  {\rm diam}_{d_{Y}}^{-1}(Y) d_{Y})$$ is injective and continuous.  However,  because of the collapsing phenomenon,  the algebro-geometric  structure of $\overline{ \mathcal{NCY}( \mathcal{M}^{P} )}$ is unclear, and it is not a compactification in the usual   algebraic geometry sense. The Gromov-Hausdorff  compactification is studied for the moduli spaces of compact  Riemann surfaces  and Abelian varieties in a recent preprint \cite{Od2}.

 Let   $( \mathcal{X} \rightarrow \Delta, \mathcal{L})$ be   a degeneration of polarized Calabi-Yau manifolds of dimension $n$   such that  the diameter of the Ricci-flat K\"ahler metric $\omega_{t}\in c_{1}(\mathcal{L}|_{X_{t}})$ tends to infinite when $t \rightarrow 0$, i.e. ${\rm diam}_{\omega_{t}}(X_{t})  \rightarrow \infty$.   Since ${\rm Vol}_{\omega_{t}}(X_{t})=\frac{1}{n!}c_{1}^{n}(\mathcal{L}|_{X_{t}})\equiv {\rm const.}$,  $(X_{t}, \omega_{t})$ must collapse (cf. \cite{An1}), i.e. for  metric 1-balls $B_{\omega_{t}}(1)$,  ${\rm Vol}_{\omega_{t}} (B_{\omega_{t}}(1))  \rightarrow 0$ when $t \rightarrow 0$.  If $0\in \Delta$ is a large complex limit  point (cf. \cite{GHJ}), a refined version of the Strominger-Yau-Zaslow (SYZ) conjecture (cf. \cite{SYZ}) due to Gross, Wilson,  Kontsevich and Soibelman (cf. \cite{GW,KS,KS2})  says that   $ {\rm diam}_{\omega_{t}}(X_{t}) \sim\sqrt{ -\log |t|}$,  and  $(X_{t}, {\rm diam}_{\omega_{t}}^{-2}(X_{t}) \omega_{t})$ converges to a compact metric space $(B,d_{B})$  in the Gromov-Hausdorff sense. If $h^{i,0}(X_{t})=0$, $1\leq i <  n$, then  $B$
 is homeomorphic to $S^{n}$.   Furthermore, there is an open subset $B_{0}\subset B$ with   ${\rm codim}_{\mathbb{R}}B\backslash B_{0} \geqslant 2$, $B_{0}$ admits  a real  affine structure, and the metric $d_{B}$ is induced by a Monge-Amp\`ere  metric $g_{B}$ on $B_{0}$, i.e.  under  affine coordinates $x_{1},  \cdots, x_{n}$,  there is a potential function $\phi$ such that $$g_{B}= \sum_{ij} \frac{ \partial^{2} \phi}{ \partial x_{i}  \partial x_{j}} dx_{i} dx_{j},    \   \  {\rm and}  \    \det (\frac{\partial^{2}\phi}{ \partial x_{i}  \partial x_{j} }) =1.$$

   This conjecture was verified by Gross and Wilson for fibred  K3 surfaces with only type $I_{1}$ singular fibers  in \cite{GW}, and was studied for higher dimensional HyperK\"ahler manifolds in \cite{GTZ, GTZ2}. Moreover, \cite{GTZ2} extends Gross-Wilson's result to all elliptically fibred K3 surfaces.  In \cite{KS2},  it was further conjectured that the Gromov-Hausdorff limit  $B$ is homeomorphic to the Calabi-Yau skeleton  of the Berkovich analytic space associated to $\mathcal{X} \times_{\Delta}\Delta^{*}$ by taking some base change if necessary, which gives an algebro-geometric description of $B$.
If we grant this version of  SYZ conjecture, we will have a nice algebro-geometric structure for the compactification at least for some one dimensional moduli space  $\mathcal{M}^{P}$.

\begin{example} A simple concrete example for Theorem \ref{main} and Theorem \ref{main2} is the mirror Calabi-Yau 3-fold of the quintic 3-fold constructed in \cite{COGP} (cf. Section 18 in  \cite{GHJ}), i.e.   $X_{t}$ is the crepant resolution of the quotient  $$Y_{s}= \{\left[z_{0}, \cdots, z_{4}\right]\in \mathbb{CP}^{4} | z_{0}^{5}+ \cdots + z_{4}^{5}+s z_{0} \cdots z_{4}=0\}/(\mathbb{Z}_{5}^{5}/\mathbb{Z}_{5})$$ of the quintic by $\mathbb{Z}_{5}^{5}/\mathbb{Z}_{5}$,  where $s^{5}= t\in \mathbb{C}$.    By choosing a polarization, $ \mathcal{M}^{P}= \mathbb{C}\backslash \{ 1\}$, and $0$ is an orbifold point of $ \mathcal{M}^{P}$.   When $t=1$,    $X_{1}$  is a Calabi-Yau variety with finite  ordinary double points, and    however, $t=\infty$ is a large complex limit point, which implies that  $t=\infty$ is the cusp end of  $ \mathcal{M}^{P}$   and   has infinite Weil-Petersson distance.
 Thus  $\overline{\mathcal{M}}^{P}=\mathbb{C}$.    Again, if we grant the refined  version of  SYZ conjecture, the point   $t=\infty$ corresponds  to  the  $S^{3}$ with a Monge-Amp\`ere  metric on an open dense subset, and consequently,     $\overline{\mathcal{M}}^{P}$  has a natural Gromov-Hausdorff  compactification $\mathbb{CP}^{1}$ with a continuous surjection  $\overline{ \mathcal{NCY}}: \mathbb{CP}^{1} \rightarrow \overline{ \mathcal{NCY}( \mathcal{M}^{P} )}$ extending $\mathcal{NCY}$.
 \end{example}

\end{document}